\documentclass[reqno, 12pt]{amsart}

\usepackage{mathrsfs}
\usepackage{amscd}
\usepackage{amsmath}
\usepackage{latexsym}
\usepackage{amsfonts}
\usepackage{amssymb}
\usepackage{amsthm}
\usepackage{graphicx}
\usepackage{hyperref}
\usepackage{makecell}
\usepackage{array}
\usepackage{booktabs}
\usepackage{multirow}
\usepackage{color,xcolor}

\usepackage{graphicx}
\usepackage{subfigure}

\newtheorem{theorem}{Theorem}[section]
\newtheorem{proposition}[theorem]{Proposition}
\newtheorem{corollary}[theorem]{Corollary}
\newtheorem{lemma}[theorem]{Lemma}

\newtheorem{remark}[theorem]{Remark}

\numberwithin{equation}{section}

\title[Inverse source problem]{%
    \textbf{Reconstruct the ambient noise source from the multi-frequency sparse correlation data}}%

\author[H. Gu]{Hao Gu}
\address{School of Mathematical Sciences, Zhejiang University, Hangzhou 310058, China}
\email{guhao@zju.edu.cn}

 \author[H. Guo]{Hongxia Guo}
 \address{School of Mathematical Sciences, and Institute of Mathematics and Interdisciplinary Sciences, Tianjin Normal University, Tianjin, 300387, China}
 \email{hxguo@tjnu.edu.cn}

\author[X. Xu]{Xiang Xu}
\address{School of Mathematical Sciences, Zhejiang University, Hangzhou 310058, China}
\email{xxu@zju.edu.cn}

 \author[Y. Zhao]{Yue Zhao}
\address{School of Mathematics and Statistics, and Key Lab NAA-MOE, Central China Normal University,
Wuhan 430079, China}
\email{zhaoy@ccnu.edu.cn}


\begin{document}

\begin{abstract}
In this paper, we develop a novel multi-frequency factorization method to reconstruct the spatial support of the ambient noise source.
The proposed method only requires sparse correlation data and has low computational cost. Numerical experiments in two  and three
dimensions are presented to demonstrate the effectiveness of the proposed method.
\end{abstract}

\maketitle

%
%
%

\section{Introduction}
Passive imaging with ambient noise aims at extracting information about unknown media or reflectors by analyzing wave fields generated by uncontrolled noise sources. A central idea of this approach is that, assuming the ambient noise is randomly distributed in space and statistically stationary in time, the cross correlation of recorded signals converges to the Green's function of the underlying wave operator. This principle forms the foundation of classical passive imaging and has led to successful applications in helioseismology, seismology, and related fields; see, for example, \cite{BGP, G_SIAP, GP_imag, GP}.

Inverse source problem concerns determining the radiating source from measurements of the wave fields away from its support, which is an important research subject in inverse scattering theory \cite{AG, zhou, blz, bao2016inverse}. In passive imaging, the stationary uncontrolled noise source is usually assumed to be mean-zero with
covariance of the form $F(t-s)K(x)\delta(x-y)$.
Using the passive imaging techniques, the inverse problem of determining
the spatial support of the noise source from the boundary cross correlation was addressed by \cite{AG}.
The strategy is to build imaging functionals which utilize the empirical cross correlation to qualitatively locate the noise source. The passive imaging methodology
relies on the assumption that the source is stationary in time. Under this assumption, the temporal stationarity leads to a diagonalization of the noise covariance operator in the frequency domain, allowing the time-domain wave equation to be decomposed into independent single-frequency Helmholtz equations. Based on this decoupling, Helmholtz--Kirchhoff identities can be employed to develop backpropagation-based imaging methods. For quantitative numerical methods for reconstructing the white noise random
sources of time-harmonic acoustic waves, we refer the reader to \cite{zhou, bao2016inverse}.

In some practical scenarios, the noise sources may be active only during finite time intervals or exhibit time-dependent intensities. Hence, the source is not stationary.
In this work, we consider the acoustic wave equation in $\mathbb{R}^d$, $d=2,3$,
\begin{equation*}
\begin{cases}
\partial_t^2 U(x,t) - \Delta U(x,t) = S(x,t), & (x,t)\in \mathbb{R}^d\times (0,\infty),\\
U(x,0)=\partial_t U(x,0)=0, & x\in\mathbb{R}^d,
\end{cases}
\end{equation*}
where $S(x,t)$ denotes a random noise source.
We assume that $S(x,t)$ is a mean-zero stochastic process
with covariance function
\begin{equation}\label{eq:cov1}
C(S(x,t)S(y,s))
= F(t-s)\, K(x,t)\,K(x,s)\,\delta(x-y),
\end{equation}
where the temporal correlation function $F(t-s)$ depends only on the time lag. $K(x,t)$ is a deterministic function describing the time-space support and intensity of the source. We assume that $K(x,t)$ has a compact support  $D\times[t_1,t_2]$, where $D\subset\mathbb{R}^d$ is a bounded Lipschitz domain  such that $\mathbb R^d\backslash D$ is connected and $0<t_1<t_2<\infty$. The time interval $[t_1,t_2]$ represents the radiating period of the time-dependent source and $D$ characterizes the time-invariant shape and location.


The goal of this paper is to reconstruct the spatial support of the random source from the correlation data. It should be noticed that,
although the temporal correlation function $F(t-s)$ is stationary, the time dependence of $K(x, t)$ implies that the random source is, in general, not stationary in time. As a consequence, the classical passive imaging theory based on temporal stationarity and Helmholtz--Kirchhoff identities may not be directly applicable. On the other hand, in many practical applications, it is required that the measurement points are limited.
For instance, only  a sparse network of seismometers are available to detect the location of the uncontrolled
earthquake \cite{GP}. Such a requirement naturally leads to inverse problems with limited
measurement positions.
To the best of our knowledge, inverse source problems driven by temporally non-stationary random sources remain largely
unexplored in the existing literature.

To address the inverse source problem and fulfill the requirement of using only limited measurement points, we employ the multi-frequency based factorization method.
The multi-frequency data can be achieved by Fourier transforming the time-domain recordings.
Specifically, letting
\[
u(x,k)=\int_{0}^{\infty} U(x,t)e^{\mathrm{i}kt}\,\mathrm{d}t,
\qquad
f(x,k)=\int_{t_1}^{t_2} S(x,t)e^{\mathrm{i}kt}\,\mathrm{d}t.
\]
We have that $u(\cdot,k)$ satisfies the inhomogeneous Helmholtz equation
\begin{equation}\label{eq:helm}
-\Delta u(x,k) - k^2 u(x,k) = f(x,k), \quad x\in\mathbb{R}^d.
\end{equation}
Notice that the resulting random source is frequency-dependent, and
the non-stationary temporal behavior of the time-domain source is encoded in the frequency dependence of its second-order statistics.

Let $G_d(x, t)$ be the fundamental solution to the wave equation with initial condition $G_d(x, t)=0$ for all $t<0$, where
$G_2(x, t) = \frac{H(t-|x|)}{2\pi\sqrt{t^2 - |x|^2}}$ with $H$ denoting the Heaviside function, and $G_3(x, t) = \frac{\delta(t-|x|)}{4\pi|x|}$.
The Fourier transform of the time-dependent Green function is given by
\[
\hat{G}_d(x, k) = \int_0^\infty G_d(x, t)e^{\mathrm{i}kt}{\rm d}t,
\]
which has the explicit form $\hat{G}_2(x, k) = \frac{i}{4}H_0^{(1)}(k|x|)$ and $\hat{G}_3(x, k) = \frac{e^{\mathrm{i}k|x|}}{4\pi|x|}$. Here $H_0^{(1)}$ is the Hankel function of the first kind
of order zero.
Therefore, the solution $u$ to the Helmholtz equation can be represented by
$$u(x,k) = \int_{\mathbb R^d}\hat{G}_d(x-y, k)f(y, k){\rm d}y.$$
Denoting $\mathbb{S}^{d-1}:=\{x\in\mathbb R^d: |x|=1\}$,
due to the asymptotic behavior of the Green function, the solution satisfies the Sommerfeld radiation condition
\begin{align}\label{src}
\lim_{r\to\infty}r^{\frac{d-1}{2}}(\partial_r u - iku) = 0, \quad r=|x|,
\end{align}
which holds uniformly in all directions $\hat{x} = x/|x|\in\mathbb S^{d-1}$.  Further,
$u$ admits the asymptotic behavior at infinity
\begin{align*}
u(x,k) = \frac{e^{\mathrm{i}k|x|}}{|x|^{\frac{d-1}{2}}}\left\{u^\infty(\hat{x},k) + \mathcal{O}(r^{-\frac{d+1}{2}})\right\}, \quad \text{as}\,\,|x|\to\infty.
\end{align*}
Here $u^\infty(\cdot,k)\in C^\infty(\mathbb S^{d-1})$ is the far-field pattern associated with the scattering problem \eqref{eq:helm}--\eqref{src} with explicit form
\begin{equation}\label{eq:far}
u^\infty(\hat{x},k)
=\int_{\mathbb{R}^d} e^{-\mathrm{i}k\hat{x}\cdot y} f(y,k)\,\mathrm{d}y=\int_{D}e^{-{\rm i} k\hat x\cdot y}\int_{t_1}^{t_2} S(y,t)e^{{\rm i}kt} {\rm d}t {\rm d}y.
\end{equation}
Under the setting in the frequency domain, the inverse problem is to determine the spatial support of $K(x,t)$, or equivalently the support of the random source $S(x,t)$, from multi-frequency far-field random data
\[
\{u^\infty(\hat{x}_j,k): k\in(k_{\min},k_{\max}),\; j=1,\dots,J\},
\]
where $\hat{x}_j\in\mathbb{S}^{d-1}$ are finitely many, possibly sparse, observation directions.

The factorization method was originally introduced by Kirsch
in the context of deterministic time-harmonic inverse scattering problems \cite{Kir98, Kir}.
Its multi-static formulation, based on far-field operators constructed from multi-incident
and multi-receiver data, provides a powerful non-iterative tool for reconstructing
geometric information about unknown scatterers or sources.
A multi-frequency factorization method
was proposed by Griesmaier \cite{Roland2017}, which allows one to recover the $\Theta-$convex polygon of the support from the multi-frequency far-field data over sparse observation directions.
However, in most existing literature, it is required that the scatterers are frequency-independent.
Recently, a novel factorization method has been developed in \cite{GH2024} which is able to deal with inverse source problems in the time domain.
Motivated by \cite{GH2024},  to locate the random source we aim to develop a multi-frequency factorization method without knowing the temporal correlation function $F$. Compared with its deterministic counterpart, the current inverse random source problem addresses a fundamentally different inverse problem with  space-time random sources and  exploits the expectation of the correlated
far-field pattern. The reconstruction is therefore performed at the covariance level rather than on single realizations, which requires a new operator formulation and a nontrivial extension
of the deterministic factorization framework.

To implement the factorization method, a far-field operator is constructed from the expectation of correlated far-field patterns,
which encodes the second-order statistical information of the random source.
Then we show that this correlation-based far-field operator
admits a symmetric factorization and satisfies a suitable range identity,
thereby enabling the reconstruction of the spatial support of the source.
A notable advantage of the associated multi-frequency factorization framework is that it requires measurements from only finitely many observation directions, in contrast to the passive imaging method used in \cite{AG} which relies on full-aperture boundary measurements.
Indeed, for a fixed observation direction,
we derive a criterion for characterizing the smallest strip
containing the source support and perpendicular to the observation direction.
Therefore, by combining data from several observation directions,
a convex polygonal approximation of the support can be recovered.
Our proposed imaging algorithm is non-iterative, computationally efficient and well-suited for sparse measurement configurations.


The remainder of this paper is organized as follows.
Section~2 introduces the far-field operator, presents its decomposition, and constructs the corresponding indicator function.
Section~3 validates the theoretical findings through numerical experiments in both two and three dimensions.
The factorization method using near-field random data is discussed in the Appendix.

\section{Factorization for far-field case}

We propose the factorization method based on the random multi-frequency far-field pattern.
As the noise source $S(x,t)$ is a mean-zero stochastic process, the far-field pattern $\eqref{eq:far}$ is also a mean-zero random field. Thus,
for each fixed observation direction $\hat x\in \mathbb{S}^{d-1}$, we introduce the covariance function
\begin{equation}\label{eq:cov}
     {C}(\hat x,\tau, k)={C}(u^{\infty}(\hat x, \tau+k), {u^{\infty}(\hat x, k)}):=\mathbb{E}(u^{\infty}(\hat x, \tau+k)\overline{u^{\infty}(\hat x, k)}),
\end{equation}
{which captures the accessible statistical information contained in the random multi-frequency far-field data.} Substituting the expression of far-field pattern \eqref{eq:far} into \eqref{eq:cov} yields
\begin{eqnarray}\label{eq:cov-g}
 &&{C}(\hat x,\tau, k)=\mathbb{E}(u^{\infty}(\hat x, \tau+k)\overline{u^{\infty}(\hat x, k)})\nonumber\\
&=&\int_{D} \int_{t_1}^{t_2}\int_{t_1}^{t_2} F(t-s)K(y,t)K(y,s)e^{{\rm i}(k+\tau)t}e^{-{\rm i} ks} \,{\rm d} t \,{\rm d}s \, e^{-{\rm i}\tau\hat x\cdot y}{\rm d}y\nonumber\\
&=&\int_{D} g(y,\tau,k)\, e^{-{\rm i}\tau\hat x\cdot y}{\rm d}y,
\end{eqnarray}
where the function $g(\cdot,\tau,k)$ is given by
\[
g(y,\tau,k)=\int_{t_1}^{t_2}\int_{t_1}^{t_2} F(t-s)K(y,t)K(y,s)e^{{\rm i}(k+\tau)t}e^{-{\rm i} ks} {\rm d}t {\rm d}s .
\]
The identity \eqref{eq:cov-g} provides the basis for defining a far-field correlation operator and developing a factorization framework for the inverse problem.
We are concerned with identifying the position and shape of the spatial  support $D$ of $K(x,t)$ from the data set of the covariance of the multi-frequency far-field data at sparse observation directions, which is
$$\{C(\hat x_j, \tau, k): \mbox{some fixed} \;k>0, \tau\in [\tau_{\min},\tau_{\max}],  j=1,2,...,J\}$$ for some $J\in \mathbb N$.

\subsection{Definition of far-field correlation operator and factorization}
Based on the covariance representation \eqref{eq:cov-g} and  the multi-frequency far-field operator proposed by \cite{Roland2017,GH2024}, we introduce the central frequency $\tau_c$ and half of the bandwidth of the given data as
\[
\tau_c:=\frac{\tau_{\min}+\tau_{\max}}{2},\quad \tau_w=\frac{\tau_{\max}-\tau_{\min}}{2}.
\]
For a given reference frequency $k=k_0>0$, we define a far-field correlation operator
\begin{align}\label{FO}
\mathcal F: L^2(0,\tau_w)\rightarrow L^2(0,\tau_w),
\end{align}
which is given by
\begin{eqnarray*}
    &&(\mathcal{F}\phi)(\tau):=(\mathcal{F}^{(\hat{x},k_0)}\phi)(\tau)   \\
    &=&\int_{0}^{\tau_w} C(\hat{x}, \tau_c+\tau-\gamma,k_0)\,\phi(\gamma)\,\mathrm{d}\gamma, \quad {\tau \in (0,\tau_w)}, \\
     &=&\int_{0}^{\tau_w} \int_D g(y,\tau_c+\tau-\gamma,k_0 )  e^{-{\rm i}(\tau_c+\tau-\gamma)\hat x\cdot y}\,\phi(\gamma){\rm d}y\,\mathrm{d}\gamma, \\
    &=& \int_{0}^{\tau_w } \int_{D} \int_{t_1}^{t_2}\int_{t_1}^{t_2} F(t-s)K(y,t)K(y,s)e^{{\rm i}(k_0+\tau_c+\tau-\gamma)t}e^{-{\rm i} k_0s} \, {\rm d} t {\rm d}s\, e^{-{\rm i}(\tau_c+\tau-\gamma)\hat x\cdot y}{\rm d}y\,\phi(\gamma)\,{\rm d}\gamma.
\end{eqnarray*}

We proceed to investigate the symmetry properties of the far-field operator $\mathcal F$. Under the covariance-based formulation introduced above, the operator can be decomposed in a symmetric manner, which reveals its intrinsic self-adjoint structure and enables a rigorous factorization analysis. We state and prove a factorization result for the far-field operator in the following proposition.
\begin{proposition}\label{Fac-F}
It holds $\mathcal{F}=\mathcal{LTL^*}$, where $\mathcal{L}=\mathcal{L}_D^{(\hat{x},k_0)}:  L^2(D\times[t_1,t_2])\rightarrow  L^2(0,\tau_w)$ is defined by
\begin{eqnarray} \label{eqn:L}
    (\mathcal{L}u)(\tau)=\int_{t_1}^{t_2}\int_D e^{-{\rm i} \tau (\hat{x}\cdot y-t)} u(y,t)\,\mathrm{d}y \mathrm{d}t,\qquad \tau\in (0, \tau_w ),
\end{eqnarray}
for all $u\in L^2(D\times[t_1,t_2])$, and
$\mathcal{T}: L^2(D\times[t_1,t_2]) \rightarrow L^2(D\times[t_1,t_2])$ is the following  multiplication operator
\begin{eqnarray}
    (\mathcal{T}\psi)(y,t):=\int_{t_1}^{t_2} F(t-s)K(y,t)K(y,s)e^{{\rm i} ((k_0+\tau_c) t-k_0 s)} \, {\rm d}s\, e^{-\mathrm{i}\tau_c \hat x \cdot y}\, \psi(y,t).   \nonumber
\end{eqnarray}
Furthermore, the operator $\mathcal L: L^2(D\times[t_1,t_2])\rightarrow L^2(0,\tau_w)$ is compact with a dense range.
\end{proposition}

\begin{proof}
 We first show that the adjoint operator $\mathcal{L}^*: L^2(0,\tau_w) \rightarrow L^2(D\times[t_1,t_2])$ of $\mathcal{L}$ can be
 expressed by
 \begin{eqnarray}\label{eqn:L*}
    (\mathcal{L^*}\phi)(y,t):=\int_0^{\tau_w} e^{{\rm i} \gamma (\hat{x}\cdot y-t)} \phi(\gamma)\,{\rm d}\gamma.
\end{eqnarray}
Indeed, for $\psi\in L^2(D\times[t_1,t_2])$ and $\phi \in L^2(0,\tau_w)$, it holds that
\begin{align*}
\langle \mathcal Lu, \phi \rangle_{L^2(0,\tau_w)}&= \int_{0}^{\tau_w}\left(  \int_{t_1}^{t_2}\int_D e^{-{\rm i} \gamma (\hat{x}\cdot y-t)} u(y, t)\,{\rm d}y\,{\rm d}t\right)\overline{\phi(\gamma)} {\rm d}\gamma \\
&= \int_{t_1}^{t_2} \int_D u(y, t) \left(\int_{0}^{\tau_w}  \overline{ e^{{\rm i}\gamma(\hat{x}\cdot y-t)}\,\phi(\gamma)} {\rm d}\gamma\right)\,{\rm d}y\,{\rm d}t \\
&=\langle u, \mathcal{L}^* \phi\rangle_{L^2(D\times[t_1,t_2])},
\end{align*}
which implies \eqref{eqn:L*}.  By the definition of $\mathcal{T}$,
\begin{equation*}
(\mathcal{TL}^*\phi)(y,t):=\int_{t_1}^{t_2} F(t-s)K(y,t)K(y,s)e^{{\rm i} ((k_0 +\tau_c)t-k_0 s)} \, {\rm d}s\, e^{-\mathrm{i}\tau_c\hat x \cdot y}\int_0^{\tau_w} e^{{\rm i} \gamma (\hat{x}\cdot y-t)} \phi(\gamma)\,{\rm d}\gamma.
\end{equation*}
Hence, combining \eqref{eqn:L} and  the definition of far-field operator \eqref{FO} yields
\begin{align*}
&(\mathcal{LTL}^*\phi)(\tau)\\
=&\int_{t_1}^{t_2}\int_D e^{-{\rm i} \tau (\hat{x}\cdot y-t)} \left( \int_{t_1}^{t_2} F(t-s)K(y,t)K(y,s)e^{{\rm i} ((k_0 +\tau_c)t-k_0 s)} {\rm d}s\, e^{-\mathrm{i}\tau_c\hat x \cdot y}\int_0^{\tau_w} e^{{\rm i} \gamma (\hat{x}\cdot y-t)} \phi(\gamma)\,{\rm d}\gamma  \right)\,\mathrm{d}y \mathrm dt \\
=&\int_{0}^{{\tau_w} } \int_{D} \int_{t_1}^{t_2}\int_{t_1}^{t_2} F(t-s)K(y,t)K(y,s)e^{{\rm i}(k_0+\tau_c+\tau-\gamma)t}e^{-{\rm i} k_0 s} \, {\rm d} t {\rm d}s\, e^{-{\rm i}(\tau_c+\tau-\gamma)\hat x\cdot y}{\rm d}y\,\phi(\gamma)\,{\rm d}\gamma\\
=&(\mathcal{F}\phi)(\tau).
\end{align*}
The compactness and density properties of the operator $\mathcal L$
 are established following the arguments for \cite[Lemma 2.1]{GH2024}. The proof is complete.
\end{proof}

It is noteworthy that the operator $\mathcal L$ maps a time-spatial-dependent source function supported on $D\times[t_1,t_2]$ to the multiple-frequency far-field correlation data at the far-field observation direction $\hat x$. It plays the same role as the data-to-pattern operator for multi-static inverse scattering problems at a fixed energy \cite[Chapter 1.4]{Kir}.  To establish a rigorous connection between the range of the far-field operator $\mathcal{F}$ and that of $\mathcal{L}$, we introduce the self-adjoint and positive-definite operator $\mathcal{F}_{\#} := |\operatorname{Re}\mathcal{F}| + |\operatorname{Im}\mathcal{F}|$. Here, the real and imaginary parts are defined as $\operatorname{Re}\mathcal{F} = (\mathcal{F} + \mathcal{F}^*)/2$ and $\operatorname{Im}\mathcal{F} = (\mathcal{F} - \mathcal{F}^*)/2$, respectively, ensuring that the square root $\mathcal{F}_{\#}^{1/2}$ is well-defined.

To facilitate the subsequent analysis of the operator's properties, we introduce the following  hypotheses:
\begin{itemize}
\item Hypothesis (H1): The temporal correlation function $F(t)$ is continuous and satisfies $F(t)>0$ for all $t\geq0$.
    \item Hypothesis (H2): $|K(x,t)|>0$ a.e. on $ D\times[t_1, t_2]$ and  the nodal set of $K$ has the Lebesgue measure of zero. Furthermore, for almost every given $x \in D$, the mapping $t \mapsto K(x,t)$ does not change sign on the interval $[t_1, t_2]$.
    \item Hypothesis (H3): The wave number $k_0 \ge 0$ and the time window length $ t_2 - t_1$ satisfy the following inequality:$$k_0 (t_2 - t_1) \le \pi - \delta_0,$$where $\delta_0 \in (0, \pi]$ is a positive constant independent of time and space.
\end{itemize}

\begin{remark}
Hypothesis (H3) possesses a clear physical interpretation. For a wave with wave number $k_0$, the corresponding temporal period is given by $T = \frac{2\pi}{k_0}$. Therefore, the condition $k_0(t_2-t_1) < \pi$ is strictly equivalent to $(t_2-t_1) < \frac{T}{2}$. Physically, this signifies that the signal radiating period must be less than half of the wave period.
\end{remark}

The following proposition from \cite[Theorem 2.2]{GH2024} is useful in the subsequent analysis of the factorization.
\begin{proposition}\label{app:range}
Let $X$ and $Y$ be Hilbert spaces  and let $F: Y\rightarrow Y$, $L: X\rightarrow Y$, and $\mathcal{T}: X\rightarrow X$ be bounded linear operators such that $F=L\mathcal{T}L^*$. We make the following assumptions
\begin{itemize}
\item[(i)] $L$ is compact with dense range and thus $L^*$ is compact and one-to-one.
\item[(ii)] $\operatorname{Re} \mathcal{T}$ and $\operatorname{Im} \mathcal{T}$ are both one-to-one and the operator $\mathcal{T}_{\#}=|\operatorname{Re}\mathcal{T}| +|\operatorname{Im}\mathcal{T}|: X\rightarrow X$ is coercive, i.e., there exists $c>0$ with
\[
\big\langle \mathcal{T}_{\#}\, \varphi, \varphi\big\rangle_X\geq c\,||\varphi||_X^2\quad\mbox{for all}\quad \varphi\in X.
\]
\end{itemize}
Then the operator $F_{\#}$ is positive and  the ranges of $F_{\#}^{1/2}:Y\rightarrow Y$ and  $L: X\rightarrow Y$ coincide.
\end{proposition}

We are now in a position to establish a fundamental range identity between the far-field operator $\mathcal F$ and the data-to-pattern operator $\mathcal L$, which forms the analytical basis of the indicator functions introduced in the subsequent section.

\begin{theorem}\label{thm:range}
There holds the range identity
\[
\operatorname{Range}\big(\mathcal F_\#^{1/2}\big)
=
\operatorname{Range}(\mathcal L).
\]
\end{theorem}

\begin{proof}
    We introduce the nodal set of $K$ and its $\epsilon-$neighborhood by
    \[Y:=\{x\in \overline D: K(x,t)=0\},\; Y_\epsilon:=\bigcup_{y\in Y} B_\epsilon(y,t),\]
    where $\epsilon>0$ is sufficiently small. We have $D\backslash \overline Y_{\epsilon}\neq\emptyset$ by the assumption $|K(x,t)|>0, (x,t)\in\overline D\times[t_1,t_2]$ almost surely and $|K(y,t)|\ge c_\epsilon >0$ for all $(y,t)\in D\backslash \overline Y_{\epsilon}\times[t_1, t_2]$.
Now we define the operator $\mathcal{F}_\epsilon: L^2(0,\tau_w)\rightarrow L^2(0,\tau_w)$ in the same way as $\mathcal F$ but over the domain $D\backslash \overline Y_{\epsilon}$ by
\begin{eqnarray*}
    &&(\mathcal{F}_{\epsilon}\phi)(\tau)\\
    &=&
    \int_{0}^{\tau_w} \int_{D\backslash\overline Y_{\epsilon}} \int_{t_1}^{t_2}\int_{t_1}^{t_2} F(t-s)K(y,t)K(y,s)e^{{\rm i}(k_0+\tau_c+\tau-\gamma)t}e^{-{\rm i} ks} \, {\rm d} t {\rm d}s\, e^{-{\rm i}(\tau_c+\tau-\gamma)\hat x\cdot y}{\rm d}y\,\phi(\gamma)\,{\rm d}\gamma.
\end{eqnarray*}
With this definition, $\mathcal F_\epsilon$ also admits the following symmetric factorization analogous to Proposition \ref{Fac-F}. Thus, we have $\mathcal{F}_\epsilon=\mathcal{L_\epsilon T_\epsilon L_\epsilon^*}$, where $\mathcal{L}_\epsilon:  L^2((D\backslash \overline Y_\epsilon )\times[t_1,t_2])\rightarrow  L^2(0,\tau_w)$ is defined by
\begin{eqnarray*}
    (\mathcal{L}_\epsilon u)(\tau)=\int_{t_1}^{t_2}\int_{D\backslash \overline Y_\epsilon } e^{-{\rm i} \tau (\hat{x}\cdot y-t)} u(y,t)\,\mathrm{d}y \mathrm{d}t,\qquad \tau\in (0, \tau_w),
\end{eqnarray*}
for all $u\in L^2((D\backslash \overline Y_\epsilon )\times[t_1,t_2])$, and
$\mathcal{T}_\epsilon: L^2((D\backslash \overline Y_\epsilon )\times[t_1,t_2]) \rightarrow L^2((D\backslash \overline Y_\epsilon )\times[t_1,t_2])$ is a  multiplication operator defined by
\begin{eqnarray}
    (\mathcal{T}_\epsilon \psi)(y,t):=\int_{t_1}^{t_2} F(t-s)K(y,t)K(y,s)e^{{\rm i} ((k_0 +\tau_c)t-k_0 s)} \, {\rm d}s\, e^{-\mathrm{i}\tau_c\hat x \cdot y}\, \psi(y,t),   \nonumber
\end{eqnarray}
which is restricted to $L^2((D\backslash \overline Y_\epsilon )\times[t_1,t_2])$.

We claim that the operator $\mathcal T_{\epsilon,\#}$ is coercive for any $\epsilon>0$. In fact,  the real and imaginary parts of  $\mathcal T_\epsilon$ are given by
\begin{align*}
\operatorname{Re}(\mathcal T_\epsilon) \psi(y,t) &= \int_{t_1}^{t_2} F(t-s)\, K(y,t)\, K(y,s)\, \cos((k_0+\tau_c)t-k_0s-\tau_c \hat x \cdot y) \, \mathrm{d}s \; \psi(y,t), \\
\operatorname{Im}(\mathcal T_\epsilon) \psi(y,t) &= \int_{t_1}^{t_2} F(t-s)\, K(y,t)\, K(y,s)\, \sin((k_0+\tau_c)t-k_0s-\tau_c \hat x \cdot y) \, \mathrm{d}s \; \psi(y,t).
\end{align*}
Hence, for any $\psi \in L^2((D \setminus \overline{Y}_\epsilon)\times [t_1,t_2])$,
\[
\begin{aligned}
\langle \mathcal T_{\epsilon,\#}\psi,\psi\rangle
&= \int_{D \setminus \overline{Y}_\epsilon} \int_{t_1}^{t_2}
\Big( |\operatorname{Re}(\mathcal T_\epsilon)(y,t)|
      + |\operatorname{Im}(\mathcal T_\epsilon)(y,t)| \Big)
      |\psi(y,t)|^2 \, \mathrm{d}t\,\mathrm{d}y .
\end{aligned}
\]
Recalling that $\mathcal{T}_\epsilon$ acts as a multiplication operator, and applying the triangle-type inequality $|a| + |b| \ge |a + \mathrm{i}b|$ for $a, b \in \mathbb{R}$, we obtain:$$|\operatorname{Re}(\mathcal T_\epsilon)(y,t)| + |\operatorname{Im}(\mathcal T_\epsilon)(y,t)| \ge |\mathcal T_\epsilon(y,t)|.
$$
By Hypothesis (H2), $K(y,s)$ does not change sign for $s \in [t_1, t_2]$,  we obtain
\[
\begin{aligned}|\mathcal T_\epsilon(y,t)| &= |K(y,t)| \cdot \left| \int_{t_1}^{t_2} F(t-s)K(y,s)e^{-{\rm i} k_0 s} \, {\rm d}s \right|\\
& = |K(y,t)| \cdot \left| \int_{t_1}^{t_2} F(t-s)|K(y,s)|e^{-{\rm i} k_0 s} \, {\rm d}s \right| \\
&\ge c_\epsilon \left| \int_{t_1}^{t_2} F(t-s)|K(y,s)|e^{-{\rm i} k_0 s} \, {\rm d}s \right|,\\
&= c_\epsilon \left| \int_{t_1}^{t_2} F(t-s)|K(y,s)|e^{-{\rm i} k_0( s-s_m)} \, {\rm d}s \right|, \\
&\ge  c_\epsilon \left| \int_{t_1}^{t_2} F(t-s)|K(y,s)|\cos k_0( s-s_m) \, {\rm d}s \right|,
 \\
&\ge  c_\epsilon  \int_{t_1}^{t_2} F(t-s)|K(y,s)|\cos k_0( s-s_m) \, {\rm d}s ,
\end{aligned}
\]
where  $s_m = \frac{t_1 + t_2}{2}$, and $|e^{{\rm i} ((k_0 +\tau_c)t - \tau_c \hat x \cdot y)}| = 1$ and $|e^{\mathrm i s_m}|=1$ are used.
By Hypothesis (H3), it holds that $k_0 |s - s_m| \le k_0 \frac{t_2 - t_1}{2} \le \frac{\pi - \delta_0}{2} = \frac{\pi}{2} - \frac{\delta_0}{2}.$ Furthermore, we obtain a strictly positive lower bound for $\cos k_0(s-s_m)$:
$$\cos(k_0(s - s_m)) \ge \cos\left(\frac{\pi}{2} - \frac{\delta_0}{2}\right) = \sin\left(\frac{\delta_0}{2}\right) > 0.$$
According to Hypothesis (H1), $F(t-s) \ge F_{\min} > 0$ for all $t-s \in [0, t_2 - t_1]$. Consequently, it follows that
$$\int_{t_1}^{t_2} F(t-s)|K(y,s)| \cos(k_0(s - s_m)) \, {\rm d}s \ge \int_{t_1}^{t_2} F_{\min} \cdot c_\epsilon \cdot \sin\left(\frac{\delta_0}{2}\right) \, {\rm d}s = F_{\min} c_\epsilon \sin\left(\frac{\delta_0}{2}\right) (t_2 - t_1).$$
Thus, we have the uniform lower bound
\[
|\operatorname{Re}(\mathcal T_\epsilon)(y,t)|
+|\operatorname{Im}(\mathcal T_\epsilon)(y,t)|
\ge F_{\min} c_\epsilon^2 \sin\left(\frac{\delta_0}{2}\right) (t_2 - t_1)
=: c_0 > 0.
\]
Substituting this estimate into the quadratic form yields
\[
\begin{aligned}
\langle \mathcal T_{\epsilon,\#}\psi,\psi\rangle
&\ge c_0
\int_{D \setminus \overline{Y}_\epsilon} \int_{t_1}^{t_2}
|\psi(y,t)|^2 \, \mathrm{d}t\,\mathrm{d}y \\
&= c_0 \|\psi\|^2_{L^2((D \setminus \overline{Y}_\epsilon)\times [t_1,t_2])}.
\end{aligned}
\]

\noindent Applying Proposition \ref{app:range},  we obtain the relation
\begin{equation}\label{range-epsilon}
  \operatorname{Range}\big(\mathcal F_{\epsilon,\#}^{1/2}\big)
=\operatorname{Range}(\mathcal L_\epsilon),
\end{equation}
Inspired by the idea of \cite{Ma2025}, we  have that $\mathcal F_\epsilon : L^2(0,\tau_w) \to L^2(0,\tau_w)$
converges to $\mathcal F: L^2(0,\tau_w) \to L^2(0,\tau_w)$ in the sense of  operator norm as $\epsilon \to 0$. Furthermore, there holds
\begin{equation} \label{range-fl}
\operatorname{Range}\big(\mathcal F_{\epsilon,\#}^{1/2}\big)
\to
\operatorname{Range}\big(\mathcal F_\#^{1/2}\big),
\qquad
\operatorname{Range}(\mathcal L_\epsilon)
\to
\operatorname{Range}(\mathcal L).
\end{equation}
Combining \eqref{range-epsilon} and \eqref{range-fl}  yields the asserted range identity. The proof is complete.
\end{proof}

To translate the abstract range characterization into geometric information, we introduce, for a given observation direction $\hat x \in \mathbb S^{d-1}$, the strip
\[
K_D^{(\hat x)}
:=
\bigl\{
y \in \mathbb R^d :
\inf_{z\in D} (\hat x \cdot z)
<
\hat x \cdot y
<
\sup_{z\in D} (\hat x \cdot z)
\bigr\}.
\]
Geometrically, $K_D^{(\hat x)}$ is the smallest strip containing the support $D$ and orthogonal to the observation direction $\hat x$.
For a given sampling point $y \in \mathbb R^d$, we further define the test function
\begin{equation}\label{eq:test}
\phi_{y,\epsilon}^{(\hat x)}(\tau)
=
\frac{1}{(t_2-t_1)\,|B_\epsilon(y)|}
\int_{t_1}^{t_2}
\int_{B_\epsilon(y)}
e^{-\mathrm{i}\tau\hat x\cdot z}
\,e^{\mathrm{i}\tau t}
\,\mathrm{d}z\,\mathrm{d}t,
\qquad
\tau \in (0,\tau_w),
\end{equation}
where $B_\epsilon(y)$ denotes the ball of radius $\epsilon$ centered at $y$.
This test function will be used to probe the range of $\mathcal L$ and to characterize the location of the source support.

\subsection{Indicator functions and reconstruction of the $\Theta$-convex hull}

By Picard's Theorem and Theorem~\ref{thm:range}, we introduce an indicator function based on the
spectral decomposition of the multi-frequency far-field operator.
Since the underlying range characterization follows directly from the established
factorization framework, we state the corresponding characterization theorem.

For a fixed observation direction $\hat x \in \mathbb S^{d-1}$, we define
\begin{equation}\label{eq:In}
 \mathcal I^{(\hat x)}(y)
 :=
\left[
\sum_{n=1}^{\infty}
\frac{
\bigl|\langle \phi^{(\hat x)}_{y,\epsilon}, \psi^{(\hat x)}_n\rangle_{L^2(0,\tau_w)}\bigr|^2
}{
|\lambda_n^{(\hat x)}|
}
\right]^{-1},
\qquad y \in \mathbb R^d,
\end{equation}
where $\{(\lambda_n^{(\hat x)},\psi^{(\hat x)}_n)\}_{n\ge1}$ denotes the eigensystem of the
self-adjoint operator $\mathcal F^{(\hat x)}_\#$, and
$\phi^{(\hat x)}_{y,\epsilon}$ is the test function centered at the sampling point $y$.

The indicator function $\mathcal I^{(\hat x)}$ provides a computational criterion
for determining whether a sampling point belongs to $K_D^{(\hat x)}$.
This criterion is based on whether the associated test function
$\phi_{y,\epsilon}^{(\hat x)}$ belongs to the range of the data-to-pattern operator
$\mathcal L^{(\hat x)}$, which encodes the geometric information of the source support. Specifically, we have the following result.

\begin{lemma}\label{lem:test}
(\cite{GH2024})  Let $\hat x \in \mathbb S^{d-1}$ be a fixed observation direction. Then the following
assertions hold:
\begin{enumerate}
\item
If $y \in K_D^{(\hat x)}$, there exists $\epsilon_0>0$ such that
\[
\phi^{(\hat x)}_{y,\epsilon} \in \mathrm{Range}\bigl(\mathcal L^{(\hat x)}\bigr)
\quad \text{for all } \epsilon\in(0,\epsilon_0).
\]
\item
If $y \notin K_D^{(\hat x)}$, then
\[
\phi^{(\hat x)}_{y,\epsilon} \notin \mathrm{Range}\bigl(\mathcal L^{(\hat x)}\bigr)
\quad \text{for all } \epsilon>0.
\]
\end{enumerate}
\end{lemma}
Furthermore, we have the following computational criterion:
\begin{theorem}\label{thm:indicator}
Let $\hat x \in \mathbb S^{d-1}$ be a fixed observation direction. Then the indicator function $\mathcal I^{(\hat x)}$ satisfies
\[
\mathcal I^{(\hat x)}(y) > 0 \quad \text{if}\quad  y \in \overline{K_D^{(\hat x)}},
\qquad
\mathcal I^{(\hat x)}(y) = 0 \quad \text{if} \quad y \notin \overline{K_D^{(\hat x)}}.
\]
\end{theorem}
\begin{proof}
    Applying Picard's Theorem (see, e.g.,\cite[Theorem 4.8]{CK2019}) gives $[\mathcal I^{(\hat x)}(y)]^{-1}<\infty$ if and only if $\phi^{(\hat x)}_{y,\epsilon} \in {\mathrm{Range}(\mathcal F_\#^{1/2})}$. By range identity ${\operatorname{Range}\big(\mathcal F_\#^{1/2}\big)}={\operatorname{Range}(\mathcal L)}$ in Theorem~\ref{thm:range}, we know $[\mathcal I^{(\hat x)}(y)]^{-1}<\infty$ if $\phi^{(\hat x)}_{y,\epsilon}\in {\mathrm{Range}(\mathcal L)}$. Combining Lemma~\ref{lem:test} and the above analysis yields $\mathcal I^{(\hat x)}(y) > 0$ if $y \in \overline{K_D^{(\hat x)}}.$ Similarly,   if $y \notin \overline{K_D^{(\hat x)}},$ $\phi^{(\hat x)}_{y,\epsilon} \notin {\mathrm{Range}(\mathcal L)}$ and $[\mathcal I^{(\hat x)}(y)]^{-1}=\infty$, then we have $\mathcal I^{(\hat x)}(y) = 0$.
\end{proof}

Theorem~\ref{thm:indicator} shows that, for a single observation direction,
the indicator function $\mathcal I^{(\hat x)}$ yields an exact characterization
of the strip $K_D^{(\hat x)}$ determined by the multi-frequency far-field correlation data.

\medskip
In practical applications, the wave fields are measured at only finitely many observation directions
$\{\hat x_j\}_{j=1}^J$.
To incorporate all available data, we define the combined indicator function
\begin{equation*}
\mathcal I(y)
:=
\left[
\sum_{j=1}^{J}
\frac{1}{\mathcal I^{(\hat x_j)}(y)}
\right]^{-1}
=
\left[
\sum_{j=1}^{J}
\sum_{n=1}^{\infty}
\frac{
\bigl|\langle \phi^{(\hat x_j)}_{y,\epsilon}, \psi^{(\hat x_j)}_n\rangle_{L^2(0,\tau_w)}\bigr|^2
}{
|\lambda_n^{(\hat x_j)}|
}
\right]^{-1},
\qquad y \in \mathbb R^d .
\end{equation*}
Correspondingly, we define the $\Theta$-convex hull of $D$ associated with the observation directions
$\{\hat x_j\}_{j=1}^J$ by
\begin{equation*}
\Theta_D
:=
\bigcap_{j=1}^{J}
K_D^{(\hat x_j)} .
\end{equation*}

Geometrically, $\Theta_D$ is the smallest convex set obtained as the intersection of all strips determined by the given observation directions.
By Theorem~\ref{thm:indicator}, each strip $K_D^{(\hat x_j)}$ can be uniquely determined from the multi-frequency far-field correlation data, and hence their intersection $\Theta_D$ is uniquely determined as well. In summary, we have the following result.

\begin{theorem}
The combined indicator function $\mathcal I$ satisfies
\[
\mathcal I(y) > 0 \quad \text{if } y \in \Theta_D,
\qquad
\mathcal I(y) = 0 \quad \text{if } y \notin \Theta_D.
\]
\end{theorem}

To conclude this section, we point out that the factorization with multi-frequency far-field data can be carried over to the case of using near-field data.
We discuss this extension in the Appendix.

\section{Numerical examples}
In this section, we present numerical experiments in both two and three
dimensions to demonstrate the effectiveness of the proposed
factorization-based method for identifying the support of the noise source
from multi-frequency far-field correlation data.
All numerical examples are shown to be consistent with the theoretical findings
in previous sections.

For a fixed observation direction $\hat{x}$ and a reference wavenumber $k_0$,
we consider the continuous far-field operator introduced in the theoretical
analysis, which acts on the frequency variable through an integral formulation.
In the numerical implementation, the frequency interval $(0,\tau_w)$ is
uniformly discretized as
\[
\omega_n = (n-0.5)\Delta\tau, \quad n = 1,\ldots,N,
\]
with step size $\Delta\tau = \tau_w/N$.
Applying a quadrature approximation to the continuous integral operator,
we obtain a finite-dimensional matrix representation
$\mathcal{F}^{(\hat{x},k_0)} \in \mathbb{C}^{N\times N}$, and the far-field correlation operator can be discretized by
\[
(\mathcal{F}^{(\hat{x},k_0)}\phi)(\tau_n)
\approx
\Delta\tau\,\sum_{m=1}^{N} C\!\left(\hat{x},\, \tau_c + \tau_n - \gamma_m,\, k_0\right) \phi(\gamma_m),
\]
where $\tau_n:=n\Delta \tau$ and $\gamma_m:=(m+0.5)\Delta\tau$, $m,n=1,\ldots,N$.
Equivalently, the resulting data matrix can be written in the explicit form
\[
\mathcal{F}^{(\hat{x},k_0)} := \Delta \tau
\begin{pmatrix}
C(\hat{x},\tau_c-\omega_1,k_0) & C(\hat{x},\tau_c-\omega_2,k_0) & \cdots & C(\hat{x},\tau_c-\omega_N,k_0) \\
C(\hat{x},\tau_c+\omega_1,k_0) & C(\hat{x},\tau_c-\omega_1,k_0) & \cdots & C(\hat{x},\tau_c-\omega_{N-1},k_0) \\
\vdots & \vdots & \ddots & \vdots \\
C(\hat{x},\tau_c+\omega_{N-1},k_0) & C(\hat{x},\tau_c+\omega_{N-2},k_0) & \cdots & C(\hat{x},\tau_c-\omega_1,k_0)
\end{pmatrix}.
\]

Let
\[
\bigl\{(\tilde{\lambda}_n^{(\hat{x},k_0)}, \phi_n^{(\hat{x},k_0)}) :
n = 1,\ldots,N \bigr\}
\]
denote an eigensystem of the discrete far-field data matrix
$\mathcal{F}^{(\hat{x},k_0)}$.
Based on this spectral decomposition, the associated factorization operator
admits the eigensystem
\[
\bigl\{(\lambda_n^{(\hat{x},k_0)}, \phi_n^{(\hat{x},k_0)}) :
n = 1,\ldots,N \bigr\},
\]
with eigenvalues given explicitly by
\[
\lambda_n^{(\hat{x},k_0)}
=
\bigl|
\operatorname{Re}(\tilde{\lambda}_n^{(\hat{x},k_0)})
\bigr|
+
\bigl|
\operatorname{Im}(\tilde{\lambda}_n^{(\hat{x},k_0)})
\bigr|.
\]
For any sampling point $y \in \mathbb{R}^{d-1}$, we define the discrete test
function vector $\tilde\phi_y^{(\hat{x})} \in \mathbb{C}^N$ by discretizing the
continuous test function \eqref{eq:test} with respect to the frequency variable.
The $n$-th component of the test vector is given by as $\epsilon\rightarrow0$
\[
\bigl(\phi_y^{(\hat{x})}\bigr)_n
=
\frac{1}{t_2-t_1}
\int_{t_1}^{t_2}
e^{-\mathrm{i}\tau_n\hat x\cdot y }
\, e^{\mathrm{i}\tau_n t}\,\mathrm{d}t,
\qquad n = 1,2,\ldots,N.
\]
Equivalently, the test vector can be written in the vector form
\[
\tilde\phi_y^{(\hat{x})}
=
\bigl(
{\phi}_y^{(\hat x)}(\tau_1),
{\phi}_y^{(\hat x)}(\tau_2),
\ldots,
{\phi}_y^{(\hat x)}(\tau_N)
\bigr)^\top.
\]
Then, substituting the eigensystem of $\mathcal{F}^{(\hat x, k_0)}_\#$ and the discretized test function into the indicator function \eqref{eq:In}, we can compute the truncated indicator function numerically.

Noting that the time-dependent source  $S(x,t)$ is real valued, we have $u^\infty (\hat  x, - k) = \overline{u^\infty (\hat x, k)}$ for all $k > 0$. We suppose $\tau_{\min}=0$, then the bandwidth of frequency can be extended from $(0,\tau_{\max})$ to $(-\tau_{\max},\tau_{\max})$ by $u^{\infty}(\hat x, -k)=\overline {u^{\infty}(\hat x, k)}$, thus one deduces from these new measurement data with $\tau_{\min}=-\tau_{\max}$ that $\tau_c=0$ and $\tau_w=\tau_{\max}$. Therefore, we obtain
\begin{eqnarray*}
    &&C(\hat x,  -\tau, k)\\
    &=&\int_D e^{\mathrm{i}\tau \hat x\cdot y}\int_{t_1}^{t_2}  \int_{t_1}^{t_2}  F(t-s)K(y,t)K(y,s) e^{{\rm i}(k-\tau)t}  e^{-{\rm i} ks}\mathrm{d} s\,\mathrm{d} t\,\mathrm{d} y  \nonumber\\
    &=&\overline {\int_D \overline{e^{\mathrm{i}\tau \hat x\cdot y}}\int_{t_1}^{t_2}  \int_{t_1}^{t_2}  F(t-s)K(y,t)K(y,s) \overline{e^{{\rm i}(k-\tau)t}}  \,\overline{e^{-{\rm i} ks}}\mathrm{d}s\,\mathrm{d}t\,\mathrm{d}y } \nonumber\\
    &=&\overline {\int_D {e^{-\mathrm{i}\tau \hat x\cdot y}}\int_{t_1}^{t_2}  \int_{t_1}^{t_2}  F(t-s)K(y,t)K(y,s) {e^{{\rm i}(\tau-k)t}}  e^{-{\rm i} (-k)s}\mathrm{d}s\,\mathrm{d}t\,\mathrm{d}y } \nonumber\\
     &=& \overline{\mathbb E [u^{\infty}(\hat x,\tau-k) \overline{u^{\infty}(\hat x, -k)}]}\\
     &=& C(\hat x, \tau,-k)= \overline{\mathbb E [u^{\infty}(\hat x,\tau-k) u^{\infty}(\hat x, k)]}
\end{eqnarray*}

In the following numerical experiments, the random source is modeled as $S(x,t) = \mathbb{G}(t)K(x,t)\dot W(x)$. Here we assume that $\mathbb{G}(t)$ is a mean-zero Gaussian process with covariance $F(t-s)=e^{-\frac{(t-s)^2}{0.1^2}}$, and $\dot W(x)$ is spatial white noise which can be interpreted in the distributional sense as the derivative of a Brownian sheet. It is straightforward to verify that the covariance function of $S(x,t)$ is consistent with the theoretical setting (\ref{eq:cov1}).
For the discretization of the random source, we introduce a regular grid of points $\{x_i\}_{i=1}^{N_x}$, $\{t_j\}_{j=1}^{N_t}$ with grid step $h_x$, $h_t$ covering the spatial and temporal support of $K$, respectively. Then Gaussian process $\mathbb{G}(t)$ can be discretized by a joint Gaussian distribution and spatial white noise $\dot W(x)$ can be approximated by a piecewise function with Gaussian variables $$\dot W(x)= \sum_{i=1}^{N_x} |K_i|^{-1}\chi_i(x)\int_{K_i}{\rm d} W(x)={h_x}^{-\frac{d}{2}}\sum_{i=1}^{N_x}\chi_i(x)Z_{i},$$
where $K_i$ represents a cell centered at $x_i$ with side length $h_x$, $\chi_i$ denotes the indicator function of $K_i$, and $Z_{i}$ are independent Gaussian random variables with mean 0 and variance 1.
We employ the Monte Carlo method to evaluate the expectation of correlation far-field data $C(\hat x, \tau, k)$, and we use (\ref{eq:far}) to compute the far-field pattern for each sample. The number of Monte Carlo samples is set to 10000.
\subsection{Numerical examples in $\mathbb{R}^2$}
In this subsection, we present several two-dimensional numerical experiments.
The frequency bandwidth is chosen as
\[
\tau \in (0,2\pi), \qquad
\tau_n = \frac{n\pi}{10}, \quad n=1,2,\ldots,N_\tau,\quad N_\tau=20,
\]
with the reference wavenumber $k_0=\pi/20$ and frequency step
$\Delta\tau=\pi/10$.
The observation directions are uniformly distributed on the unit circle and
given by
\[
\hat x_n = (\cos\theta_n,\sin\theta_n), \qquad
\theta_n = \frac{(n-1)\pi}{N_{\hat x}}, \quad n=1,2,\ldots,N_{\hat x},
\]
where $N_{\hat x}=12$.
For each observation direction, multi-frequency far-field correlation data are
generated and assembled to form the discrete far-field operator used in the
factorization algorithm.
The sampling region is taken as a square domain containing the support of the
source, which is uniformly discretized into a Cartesian grid.
At each sampling point, the indicator function associated with the
factorization method is evaluated.
In all examples below, the reconstructed support is visualized by plotting the
 indicator function and the boundary of the source support is highlighted by pink solid line, where higher values indicate a higher
likelihood that the sampling point belongs to the source support.

\medskip
\noindent
\textbf{Example 1 (Disk-shaped support).}
In the first example, the noise source $S(x,t)$ is supported in a disk,
\[
D=\{(x_1,x_2)\in\mathbb{R}^2:\ x_1^2+x_2^2\le 1\}.
\]
Figure~\ref{fig:2-exa1} shows the reconstruction results  using far-field data from one or 12 observation directions for a disk-shaped source support with $K(x,t)=t+1$ and different radiating time period $[0, t_2]$.
The smallest strips $K_D^{(\hat x)}$  and circular shape of the source support are clearly identified, and the boundary of the disk is well resolved.
This example demonstrates that the proposed method is capable of accurately
recovering simple convex geometries from a limited number of observation
directions and frequencies. We also observe that, as the radiation period of the noise signal increases,
the reconstruction quality slightly deteriorates, leading to a mild loss of
accuracy in both the shape and the location of the identified support.
Nevertheless, the overall geometry of the disk remains clearly recognizable.

\medskip
\noindent
\textbf{Example 2 (Kite-shaped support).}
In the second example, the source is supported in a kite-shaped domain defined
by the parametric curve
\[
x(\theta)
= (x_1,x_2)
+ \bigl(
r\cos\theta + a r (\cos 2\theta - 1),
\;
b r \sin\theta
\bigr),
\qquad \theta\in[0,2\pi],
\]
with parameters $a=0.65$, $b=1.5$, and $r=1$.

The reconstructed images using far-field data from 12 observation directions for a kite-shaped source support with different $K(x,t)$ and  radiating time period $[0, 0.1]$ are shown in Figure~\ref{fig:2-exa2}.
For subfigures (a), (c), and (e), the center of the kite-shaped support is located
at the origin, whereas in subfigures (b), (d), and (k) the support is shifted away
from the origin.
It can be observed from subfigures (a), (d), (e), and (k) that the proposed
factorization-based method is able to accurately recover the location of the
kite-shaped support as well as its convex hull, even in the presence of strong
non-convexity and sharp geometric features.
In contrast, for subfigures (b) and (c), only partial information on the support
boundary is obtained.
In these cases, the imaging functional mainly highlights segments of the support
edges, while the interior region is not fully reconstructed.
This phenomenon can be explained by the relatively weak effective signal
intensity associated with these configurations.
From an operator-theoretic point of view, the corresponding multi-frequency
far-field correlation operator exhibits faster spectral decay, which reduces
the contribution of informative singular components and consequently limits
the reconstruction capability of the factorization method.
Overall, this example demonstrates that the proposed approach can reliably
recover both the position and the convex hull of a kite-shaped support in most
configurations, while also illustrating the intrinsic dependence of the
reconstruction quality on the effective strength of the measured data.

\medskip
\noindent
\textbf{Example 3 (Rounded square support).}
The third example considers a rounded square-shaped support given by
\[
x(\theta)
= (x_1,x_2)
+ r(\cos^3\theta+\cos\theta,\ \sin^3\theta+\sin\theta),
\qquad \theta\in[0,2\pi],
\]
with $r=0.8$.
The reconstruction result is shown in Figure~\ref{fig:2-exa3}.
Although the boundary contains flat regions and rounded corners, the proposed
method still yields a stable and accurate identification of the overall shape.
However, it can be observed that the reconstruction near the four rounded
corners is less accurate. This phenomenon can be attributed to the limited angular coverage and the
relatively small number of observation directions.

\medskip
\noindent
\textbf{Example 4 (Elliptic support).}
In the final example, the noise source is supported in an ellipse described by
\[
x(\theta)
= (x_1,x_2) + (a\cos\theta,\ b\sin\theta),
\qquad \theta\in[0,2\pi],
\]
where $a=1.5$, and $b=0.8$.
Figure~\ref{fig:2-exa4} presents the corresponding reconstruction.
The elongation and orientation of the ellipse are clearly recovered, further
confirming the capability of the method to distinguish different aspect ratios
of source supports. It is worth noting that the reconstruction quality in this example is very
similar to that observed for the kite-shaped support in Example~2.
This is consistent with the fact that the signal intensity functions are
constructed in a comparable manner in both cases, leading to multi-frequency
data of similar effective strength.
As a result, the proposed method exhibits nearly identical reconstruction
performance for these two geometrically different source supports.


\subsection{Numerical examples in $\mathbb{R}^3$}

In this subsection, we present three-dimensional numerical experiments to
further validate the effectiveness of the proposed multi-frequency
factorization method.
Compared with the two-dimensional case, the three-dimensional setting poses
additional challenges due to the increased complexity of the geometry and the
limited-view nature of practical observations.

The frequency bandwidth is chosen in the same way as that used in the two-dimensional setting.
In each numerical experiment, multi-frequency data are computed for the
following three sets of observation directions:
\begin{itemize}
  \item The first set consists of three canonical directions,
  \[
  \hat x=\{(1,0,0),(0,1,0),(0,0,1)\}.
  \]
  \item The second set contains $10$ observation directions uniformly
  distributed on the upper hemisphere.
  \item The third set contains $20$ observation directions uniformly
  distributed on the upper hemisphere.
\end{itemize}
These three configurations are used to illustrate the influence of the number
of observation directions on the reconstruction quality.
The sampling region is chosen as a cubic domain containing the support of the
source and discretized by a uniform three-dimensional grid.
For each sampling point, the indicator function associated with the
factorization method is evaluated for visualization. The reconstructed and original source supports is further projected onto the coordinate planes $y_1 y_2$, $y_1 y_3$, and $y_2 y_3$, where the corresponding projection boundaries are delineated by red curves to facilitate comparison.

\medskip
\noindent
\textbf{Example 1 (Cube-shaped support).}
In the first example, the noise source $S(x,t)$ is supported in a cube (see Figure~\ref{fig:3-cube-ma}(a)),
\[
D=\{(x_1,x_2,x_3)\in\mathbb{R}^3:\ |x_1|\le 1,\ |x_2|\le 1,\ |x_3|\le 1\}.
\]
The reconstruction results for different sets of observation directions are
shown in Figure~\ref{fig:3-cube-1a} and ~\ref{fig:3-cube-ma}.
Figure~\ref{fig:3-cube-ma} illustrates the reconstructions of $K_D^{(\hat x)}$ using far-field data under different observation directions.
When data from a single observation direction $\hat x$ are employed, the reconstructed indicator clearly exhibits a slab-like structure, which is consistent with the projection of the true source support onto the coordinate planes.
In Figure~\ref{fig:3-cube-ma}(b), three carefully chosen observation directions are used, and the overall shape of the source support is well reconstructed.
In Figures~\ref{fig:3-cube-ma}(c) and~\ref{fig:3-cube-ma}(d), far-field data from $10$ and $20$ observation directions, respectively, uniformly distributed on the upper hemisphere are adopted.
In these cases, the location of the source support is accurately identified; however, the reconstructed shape appears smoother than the true cube.
This smoothing effect is mainly caused by measurement noise and data discretization, which tend to suppress high-frequency components and thus round off the sharp corners of the cube.

Figures~\ref{fig:3-cube-1b} and~\ref{fig:3-cube-mb} present the reconstruction results for the same cube-shaped source support as in Figures~5 and~6, except that the cube is no longer centered at the origin.
Figure~\ref{fig:3-cube-1b}shows the reconstruction obtained from a single observation direction, while Figure~\ref{fig:3-cube-mb} corresponds to multiple observation directions.
Compared with Figures~5 and~6, the spatial location of the translated source support is still accurately identified; however, for multiple uniformly distributed observation directions, the reconstruction of the shape becomes less accurate, with the sharp corners appearing smoother.
This effect is mainly due to the loss of geometric symmetry and the presence of noise, which reduce the sensitivity to high-frequency shape features.

It is observed that, after translating the cube away from the origin, the reconstruction obtained from multiple uniformly distributed observation directions exhibits a degradation in shape recovery, while the location of the source support remains accurately identified.
This phenomenon can be attributed to the loss of geometric symmetry and the limited observation aperture.
For an off-centered target, the uniformly distributed observation directions on the upper hemisphere no longer provide balanced angular coverage of the source support, which leads to a reduced sensitivity to high-frequency geometric features, such as sharp edges and corners.
Moreover, in the presence of noise and finite bandwidth, the multi-directional data tend to suppress high-frequency components, resulting in a smoothing effect on the reconstructed shape.

\medskip
\noindent
\textbf{Example 2 (Ball-shaped support).}
In the second example, the noise source is supported in a ball,
\[
D=\{(x_1,x_2,x_3)\in\mathbb{R}^3:\ x_1^2+x_2^2+x_3^2\le 1\}.
\]
Figures~\ref{fig:3-ball-ma} show the reconstruction results for a ball-shaped source support, where the true support is displayed in Figure~\ref{fig:3-ball-ma}~(a).
Figure~\ref{fig:3-ball-ma}~(b)-(d) present the reconstructions obtained from far-field data corresponding to $3$, $10$, and $20$ observation directions, respectively.
Figure~\ref{fig:3-ball-ma} considers a ball centered at the origin, while Figure~\ref{fig:3-ball-mb} illustrates the case with an off-origin center.
In all these configurations, the reconstructed supports are in good agreement with the true ones, and increasing the number of observation directions leads to more stable reconstructions.
Compared with the cube-shaped support, the ball-shaped support tends to exhibit slightly improved reconstruction quality, which may be attributed to its smooth boundary and the absence of sharp edges or corners, leading to a more favorable behavior of the inversion process.

\medskip
\noindent
\textbf{Example 3 (Ellipsoidal support).}
The third example considers an ellipsoidal support given by
\[
D=\{(x_1,x_2,x_3)\in\mathbb{R}^3:\ x_1^2+9x_2^2+4x_3^2\le 1\}.
\]
The reconstruction results are shown in Figure~\ref{fig:3-ellip-ma}.
Despite the anisotropic geometry of the ellipsoid, a satisfactory reconstruction
can already be achieved using data from a small number of observation directions.
As the number of observation directions increases, the elongation and orientation
of the ellipsoid are recovered with improved accuracy.
This example demonstrates that the proposed method is capable of resolving
anisotropic source supports in three dimensions, even with limited angular
information. Figures~\ref{fig:3-ellip-ma} and~\ref{fig:3-ellip-mb} show the reconstruction results for an ellipsoidal source support, corresponding to the same experimental settings as in Figures~\ref{fig:3-ball-ma} and~\ref{fig:3-ball-mb}.
The reconstructions exhibit similar behavior and comparable accuracy to the ball-shaped case.

\medskip
\noindent
The three-dimensional numerical experiments confirm that the proposed
multi-frequency factorization method remains effective in higher dimensions.
While limited observation data may lead to reduced resolution, incorporating a
larger number of observation directions substantially improves the accuracy
and stability of the reconstruction.
These results are in full agreement with the theoretical analysis and further
demonstrate the applicability of the method to practical three-dimensional
inverse source problems.

\begin{remark}
    In inverse source problems, the reconstruction of sharp geometric features
such as corners or cusps is intrinsically challenging.
This difficulty stems from the smoothing nature of the source-to-field
mapping, in which high-frequency spatial information associated with
non-smooth source supports is significantly attenuated in the measured
far-field data.
Under finite-frequency and limited-aperture measurements, such information
becomes particularly unstable, so that imaging methods tend to reliably
recover the location and convex hull of the source support, while finer
boundary details are naturally smoothed out.
\end{remark}

%
%

%
%
%
%
%
%
%
%
%
%

\begin{figure}
\centering
\subfigure[$t_2=0.1, \hat x=(1,0)$ ]{
\includegraphics[scale=0.22]{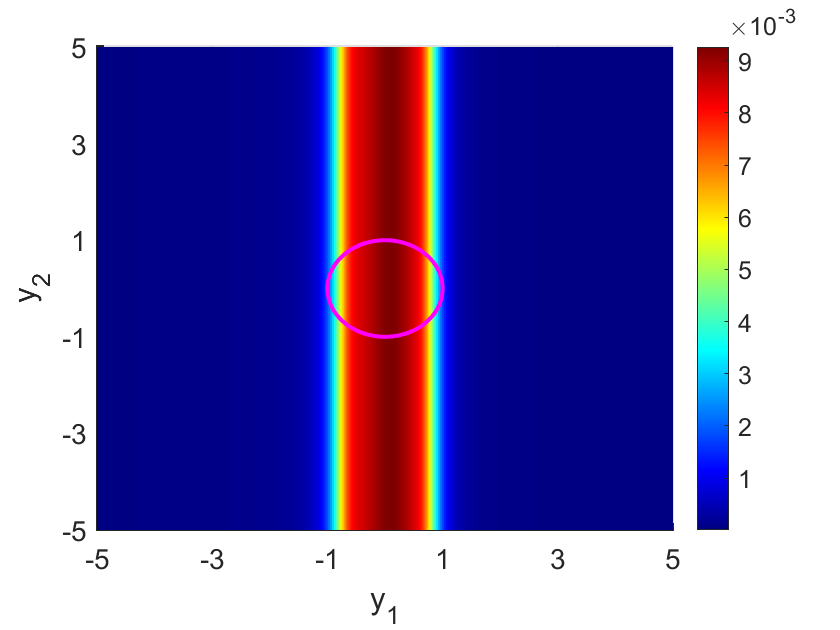}
}
\subfigure[$t_2=0.1, \hat x=(\sqrt 2/2,\sqrt 2/2)$]{
\includegraphics[scale=0.22]{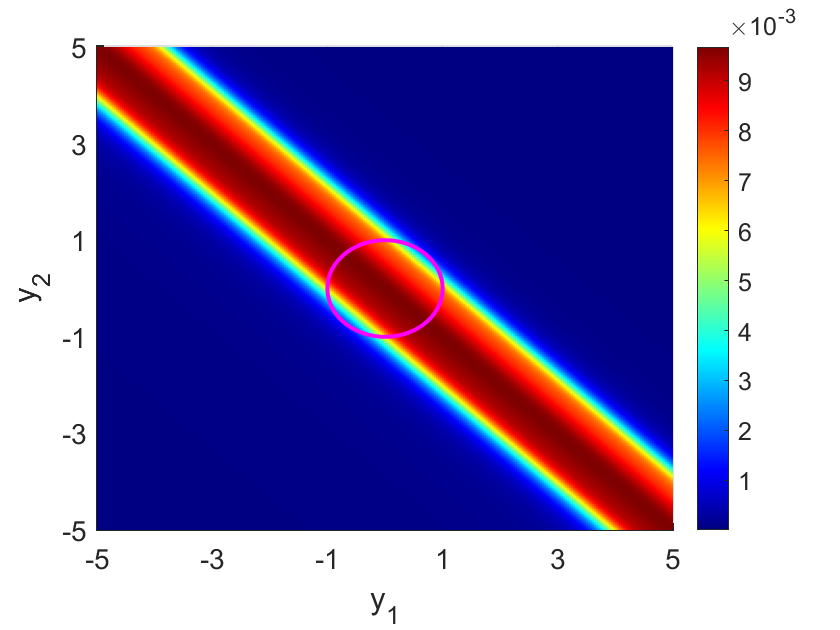}

}
\subfigure[$t_2=0.1, \hat x=(0,1)$ ]{
\includegraphics[scale=0.22]{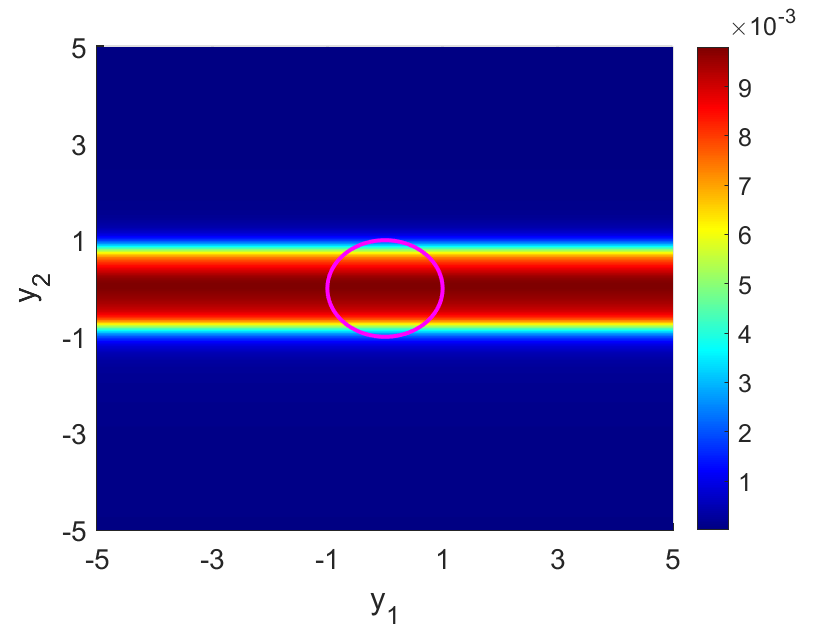}
}

\subfigure[$t_2=0.1, N_{\hat x}=12$]{
\includegraphics[scale=0.22]{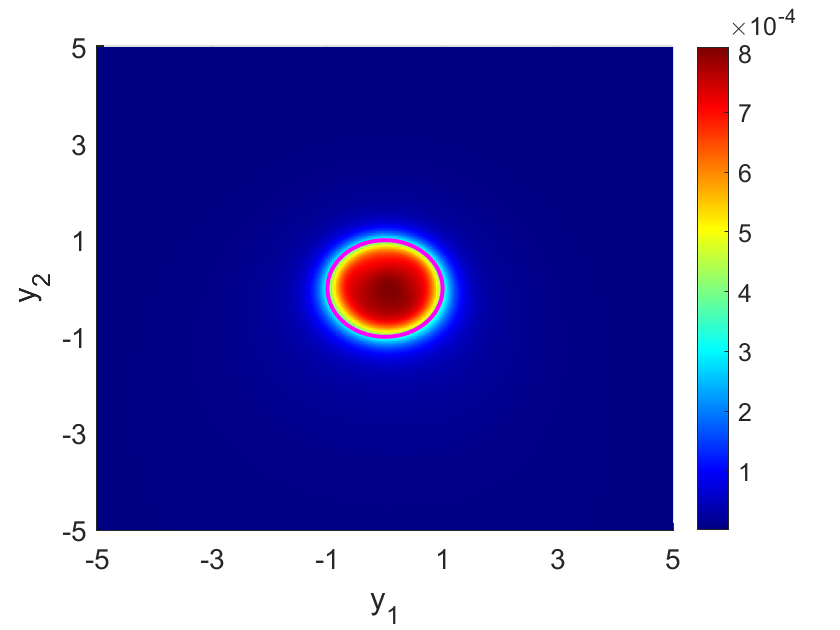}

}
\subfigure[$t_2=1, N_{\hat x}=12$ ]{
\includegraphics[scale=0.22]{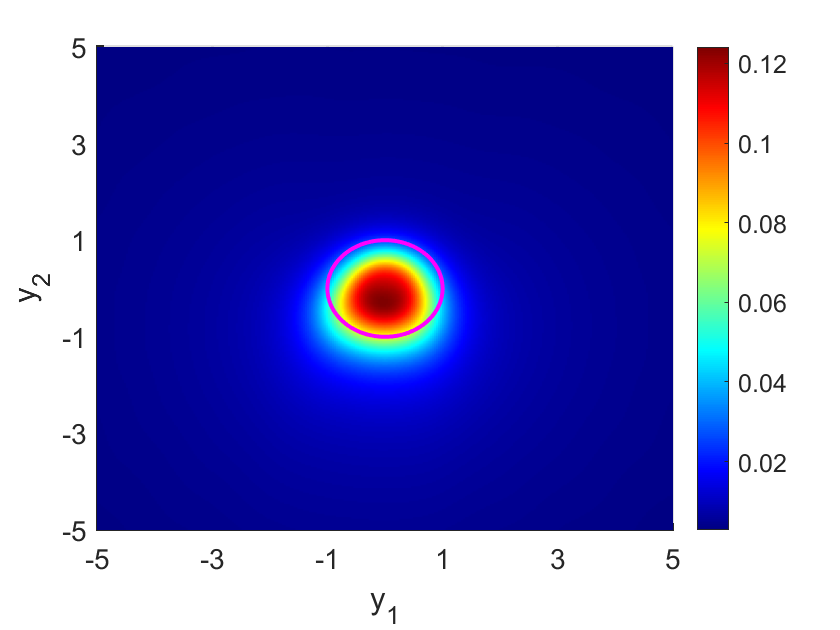}
}
\subfigure[$t_2=2, N_{\hat x}=12$]{
\includegraphics[scale=0.22]{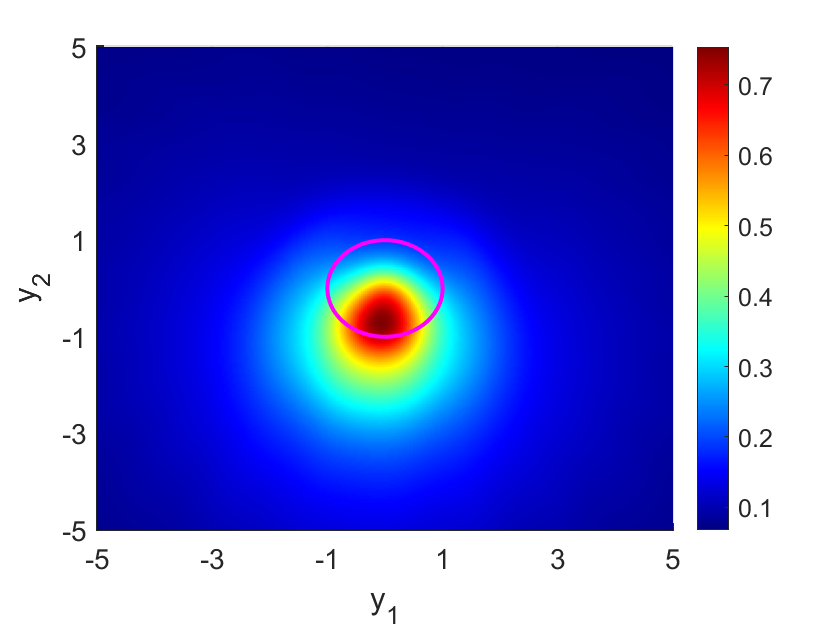}

}
\caption{Reconstructions using far-field data from one or 12 observation directions for a disk-shaped source support with $K(x,t)=t+1$ and different radiating time period $[0, t_2]$.
} \label{fig:2-exa1}
\end{figure}

\begin{figure}
\centering
\subfigure[$K(x,t)=t+1$]{
\includegraphics[scale=0.22]{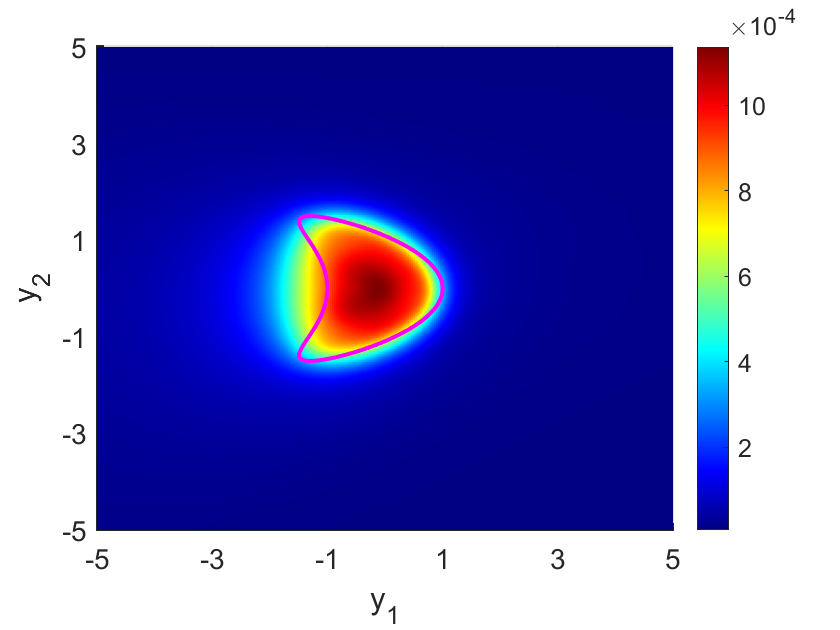}

}
\subfigure[$K(x,t)=(x_1^2+x_2^2)(t+1)$]{
\includegraphics[scale=0.22]{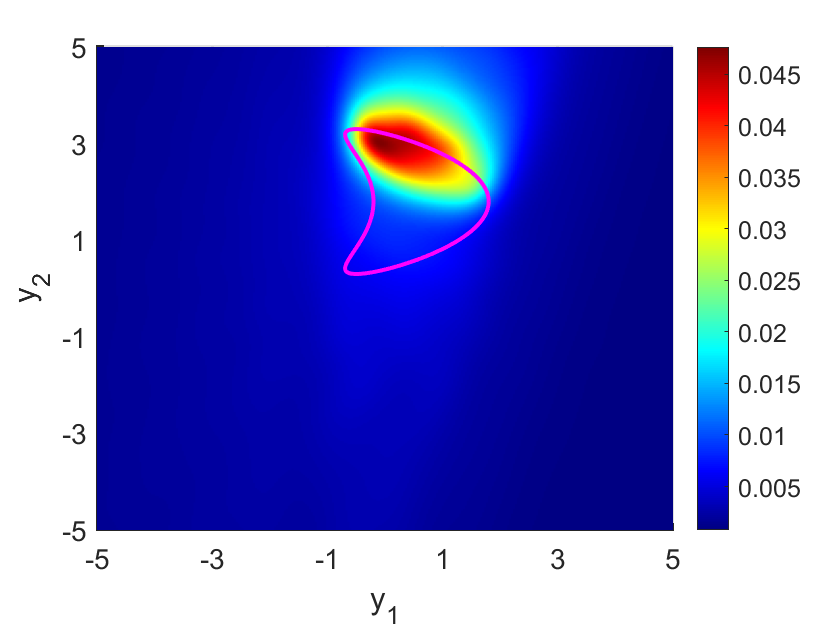}

}
\subfigure[$K(x,t)=(x_1^2+x_2^2)(t+1)$]{
\includegraphics[scale=0.22]{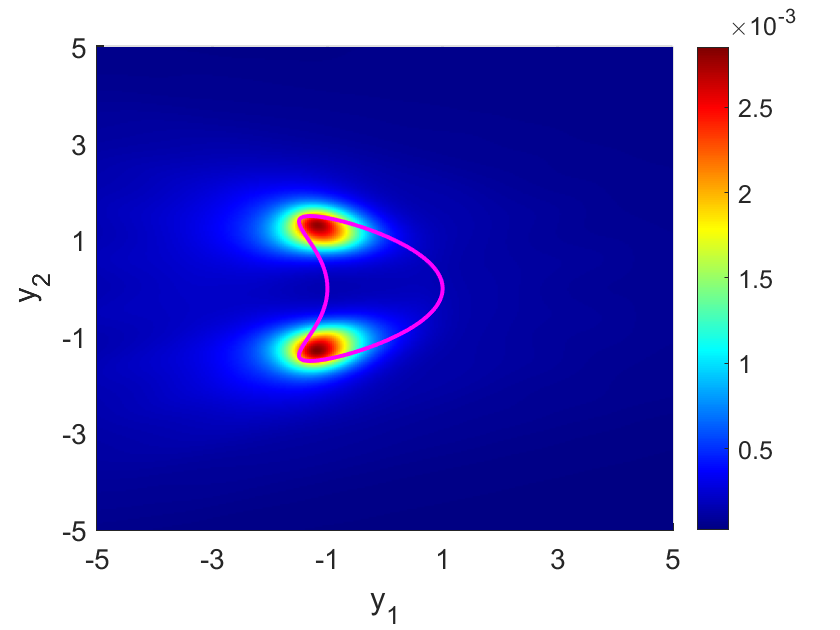}

}

\subfigure[$K(x,t)=t+1$]{
\includegraphics[scale=0.22]{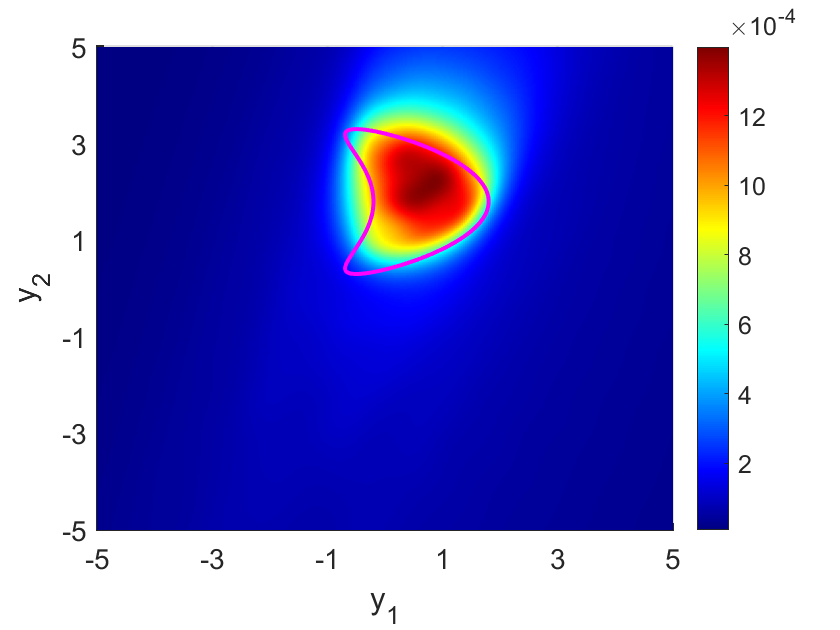}

}
\subfigure[$K(x,t)=(x_1^2+x_2^2+4)(t+1)$]{
\includegraphics[scale=0.22]{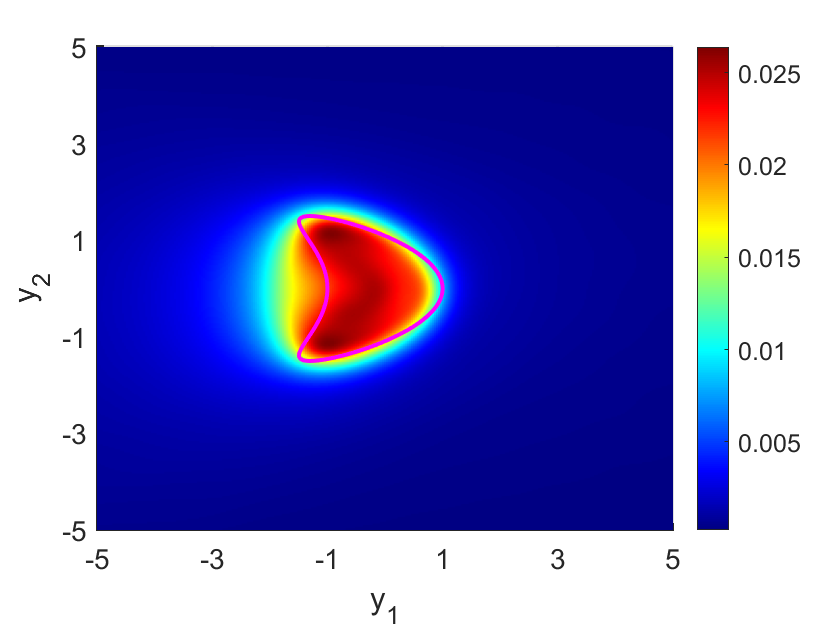}

}
\subfigure[$K(x,t)=(\sin(x_1+2x_2)+4)(t+1)$]{
\includegraphics[scale=0.22]{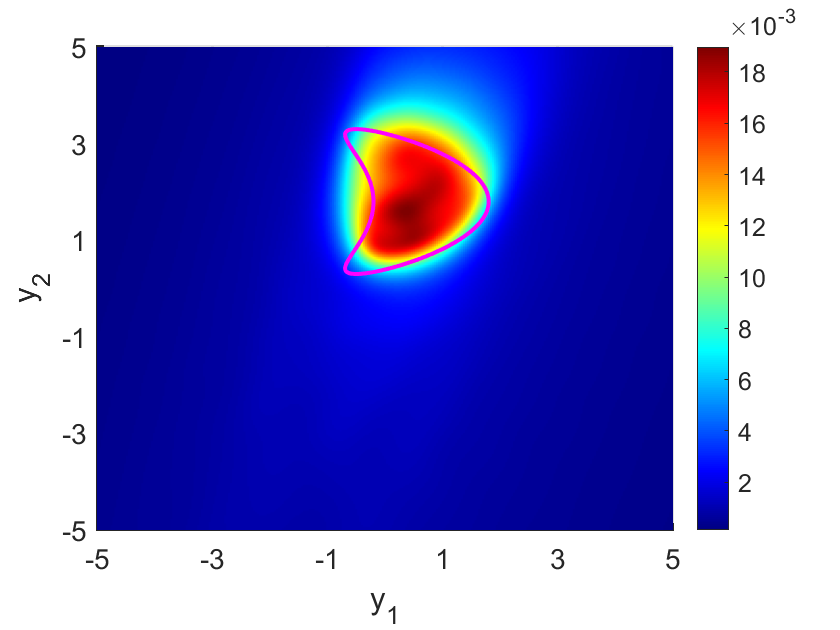}

}
\caption{Reconstructions using far-field data from 12 observation directions for a kite-shaped source support with different $K(x,t)$ and  radiating time period $[0, 0.1]$.
} \label{fig:2-exa2}
\end{figure}

\begin{figure}
\centering
\subfigure[$K(x,t)=t+1$]{
\includegraphics[scale=0.22]{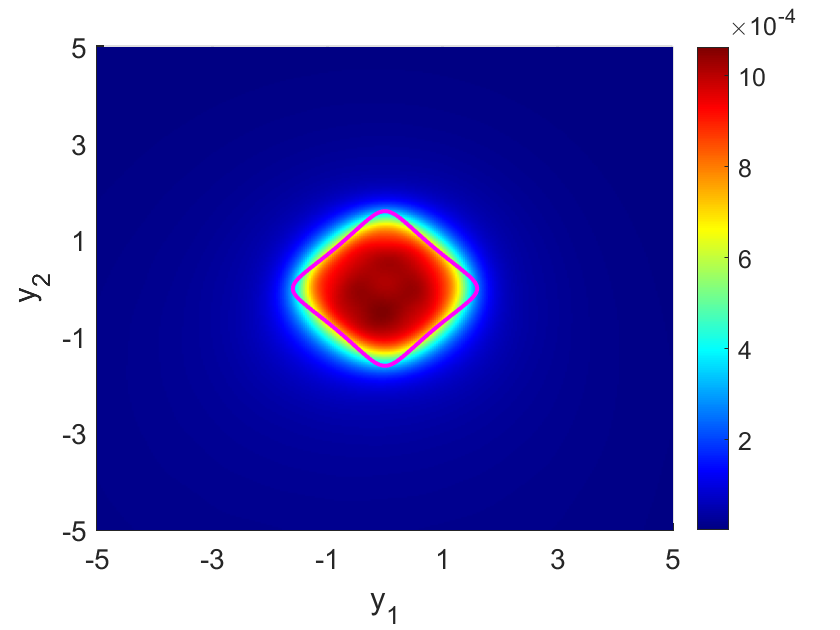}

}
\subfigure[$K(x,t)=x_1(t+1)$]{
\includegraphics[scale=0.22]{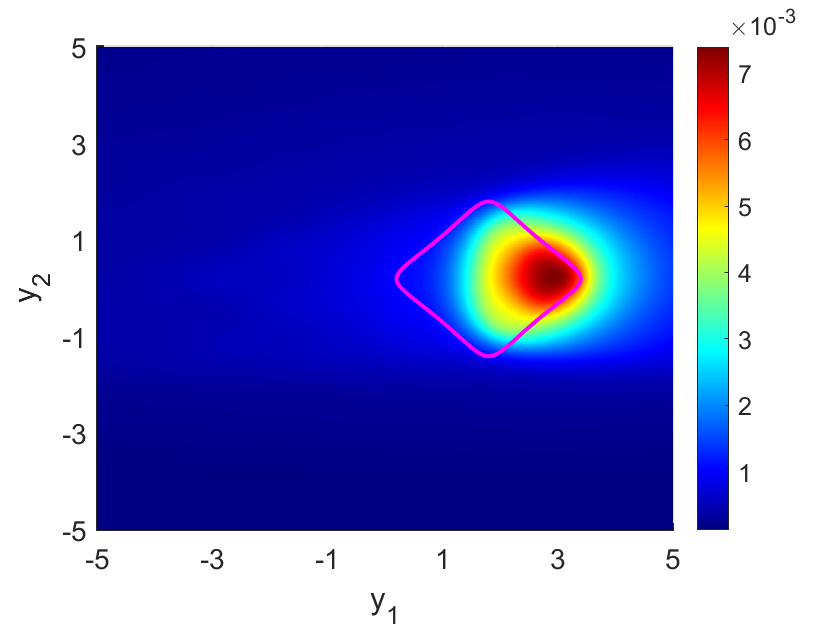}

}
\subfigure[$K(x,t)=(x_1^2+x_2^2)(t+1)$]{
\includegraphics[scale=0.22]{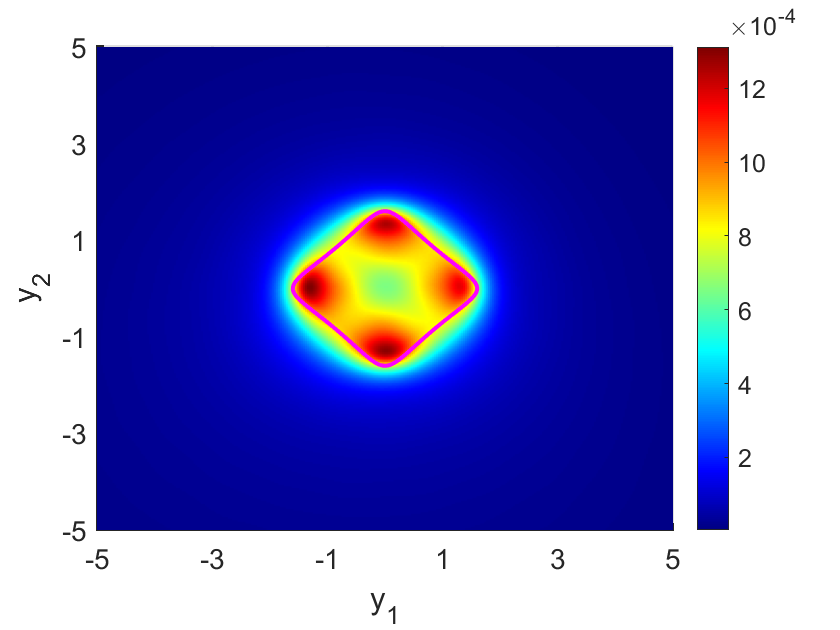}

}

\subfigure[$K(x,t)=t+1$]{
\includegraphics[scale=0.22]{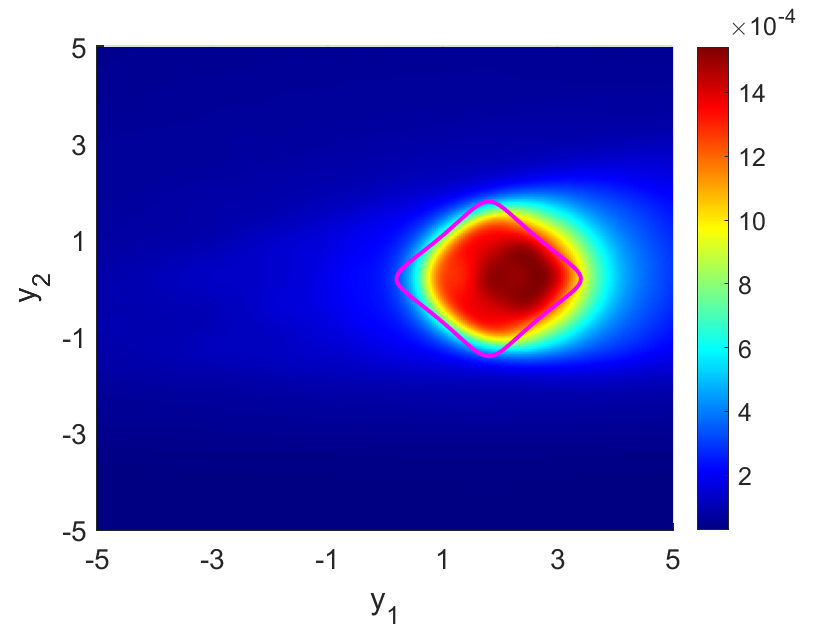}

}
\subfigure[$K(x,t)=(x_1^2+x_2^2+4)(t+1)$]{
\includegraphics[scale=0.22]{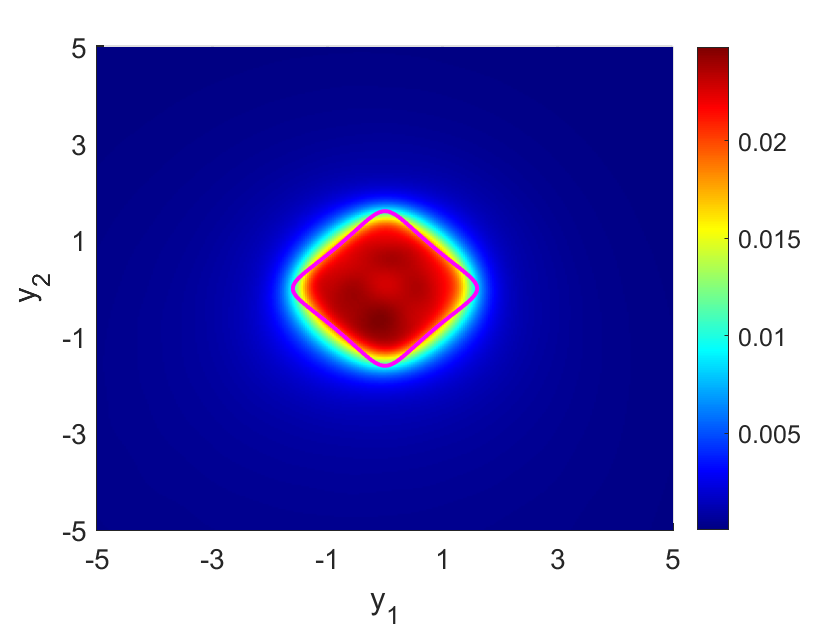}

}
\subfigure[$K(x,t)=(3-0.5((x_1-1.8)^2+(x_2-0.2)^2))(t+1)$]{
\includegraphics[scale=0.22]{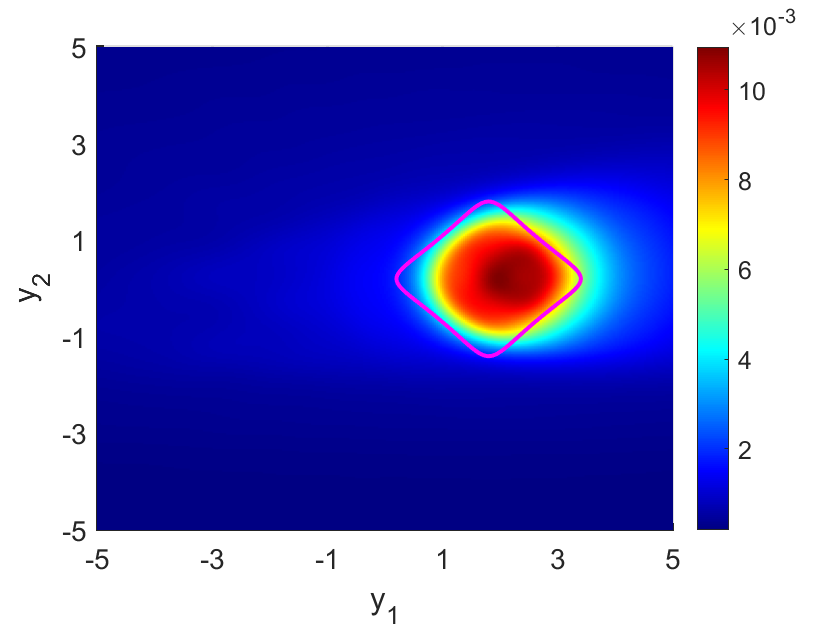}

}
\caption{Reconstructions using far-field data from using far-field data from 12 observation directions for a rounded square source support with different $K(x,t)$ and  radiating time period $[0, 0.1]$
} \label{fig:2-exa3}
\end{figure}

\begin{figure}
\centering
\subfigure[$K(x,t)=t+1$]{
\includegraphics[scale=0.22]{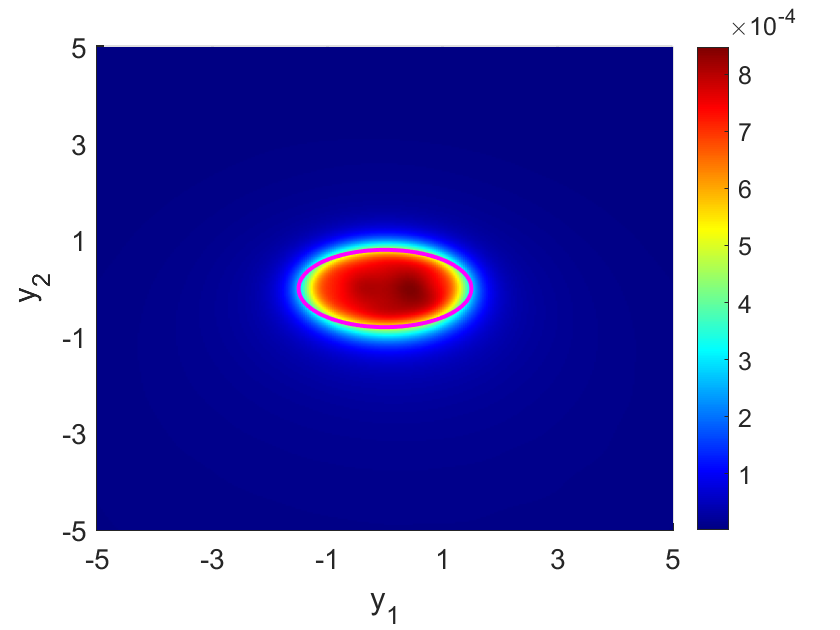}

}
\subfigure[$K(x,t)=(x_1+2x_2^2)(t+1)$]{
\includegraphics[scale=0.22]{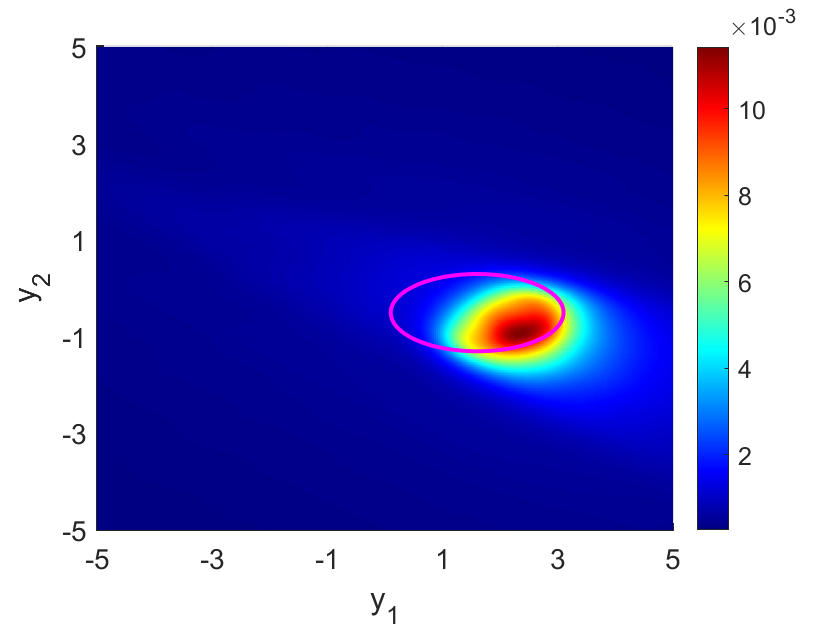}

}
\subfigure[$K(x,t)=(x_1^2+x_2^2)(t+1)$]{
\includegraphics[scale=0.22]{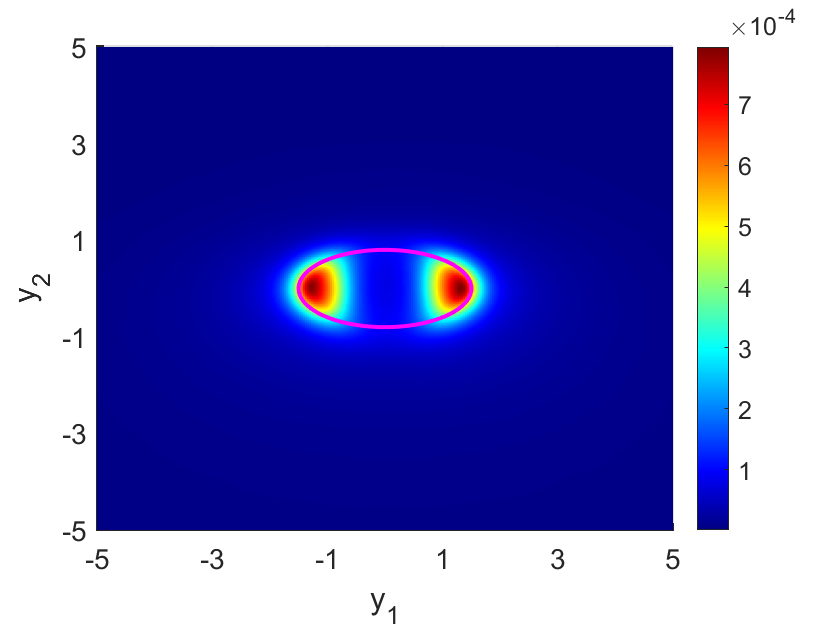}

}

\subfigure[$K(x,t)=t+1$]{
\includegraphics[scale=0.22]{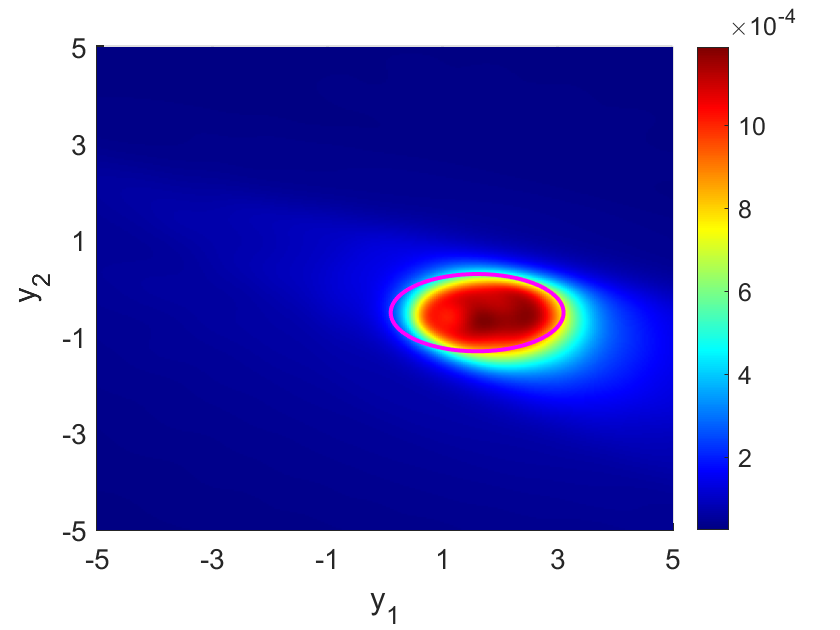}

}
\subfigure[$K(x,t)=(x_1^2+x_2^2+4)(t+1)$]{
\includegraphics[scale=0.22]{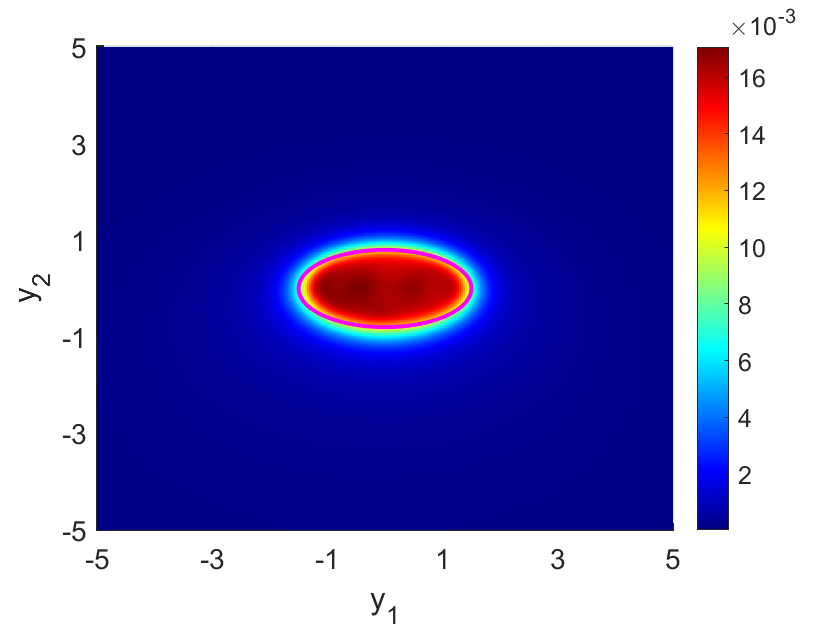}

}
\subfigure[$K(x,t)=(0.5\sin(2x_1)+\cos(x_2)+3)(t+1)$]{
\includegraphics[scale=0.22]{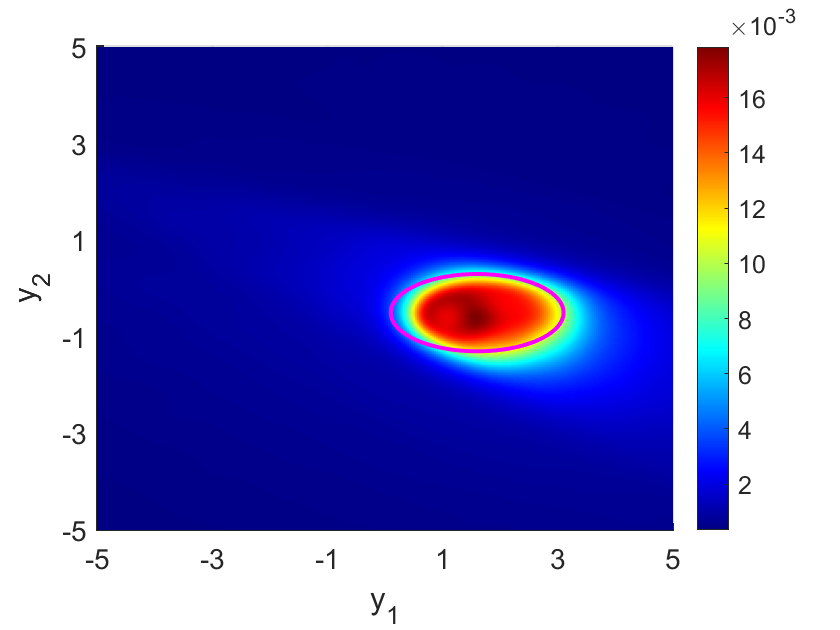}

}
\caption{Reconstructions using far-field data from 12 observation directions for an elliptic source support with different $K(x,t)$ and  radiating time period $[0, 0.1]$
} \label{fig:2-exa4}
\end{figure}


\begin{figure}
\centering
\subfigure[$\hat x=(1, 0, 0)$]{
\includegraphics[scale=0.2]{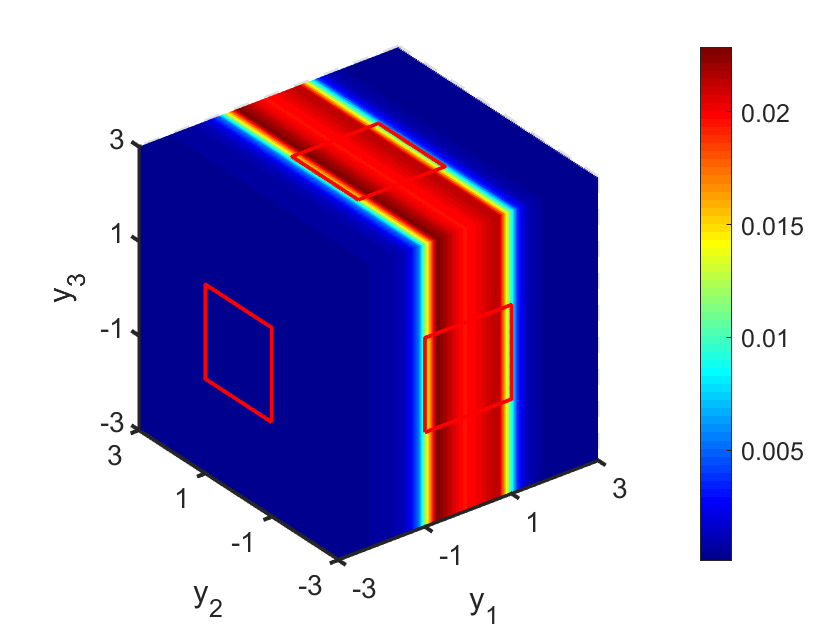}

}
\subfigure[$\hat x= (0,1,0)$]{
\includegraphics[scale=0.2]{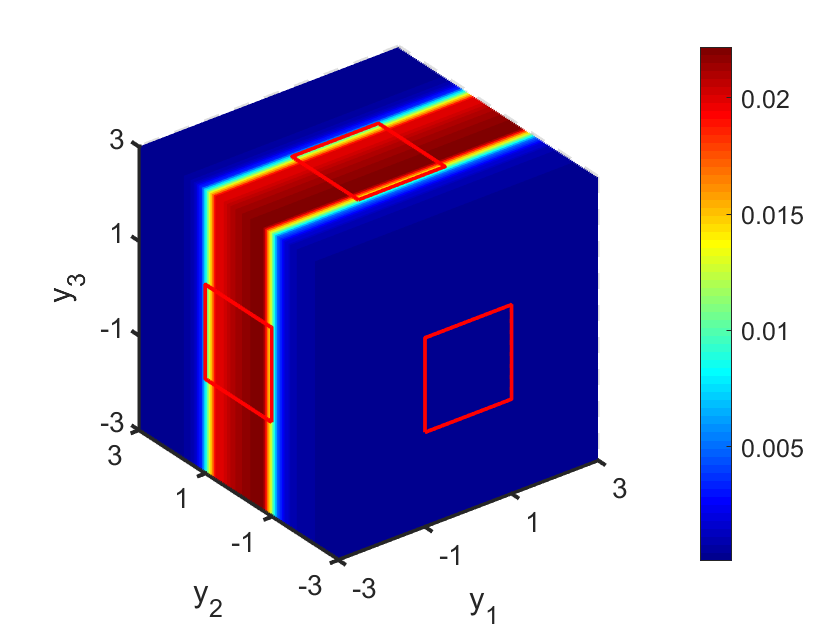}

}
\subfigure[$\hat x =(0,0, 1)$]{
\includegraphics[scale=0.2]{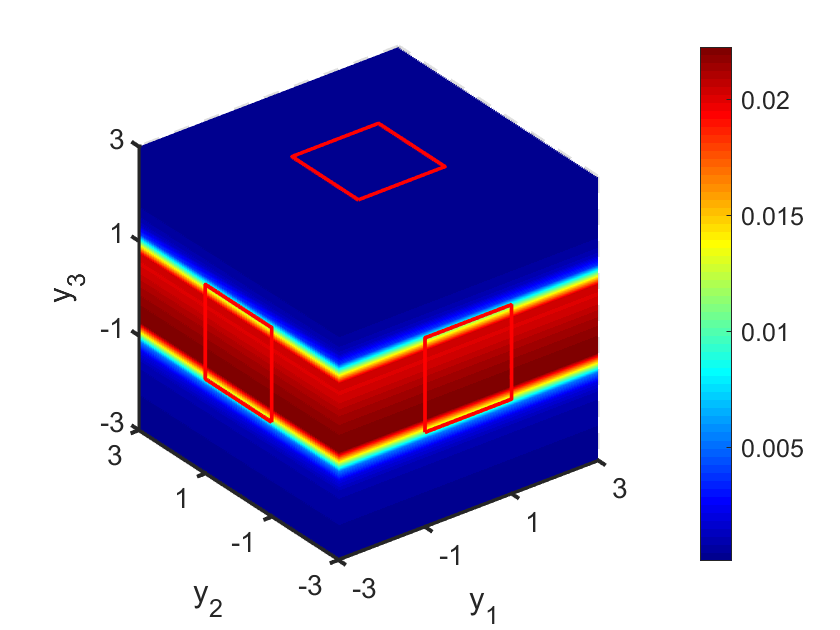}

}
\caption{Reconstructions using far-field data from a single observation direction for a cube-shaped support centered at the origin.
} \label{fig:3-cube-1a}
\end{figure}

\begin{figure}
\centering
\subfigure[Cube-shaped support]{
\includegraphics[scale=0.16]{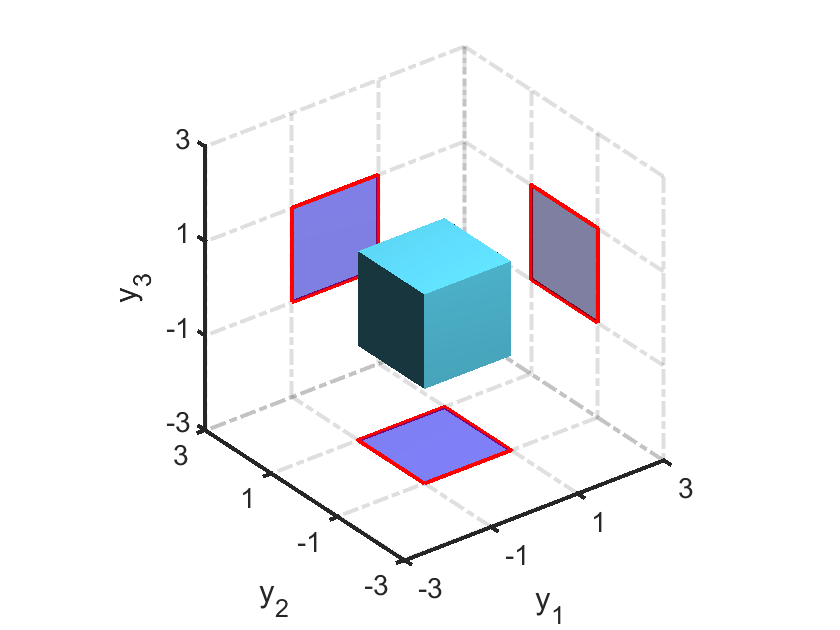}

}
\subfigure[$J=3$]{
\includegraphics[scale=0.16]{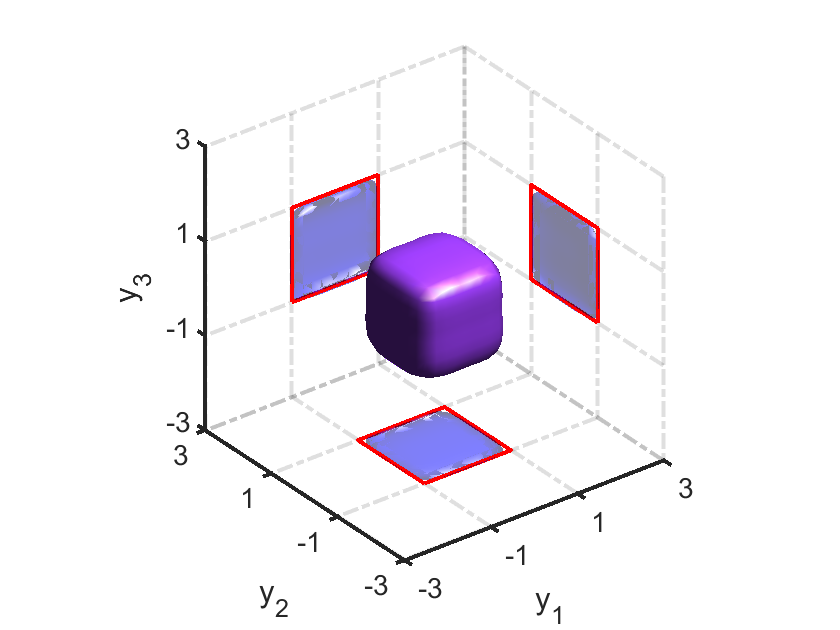}

}
\subfigure[$J=10$]{
\includegraphics[scale=0.16]{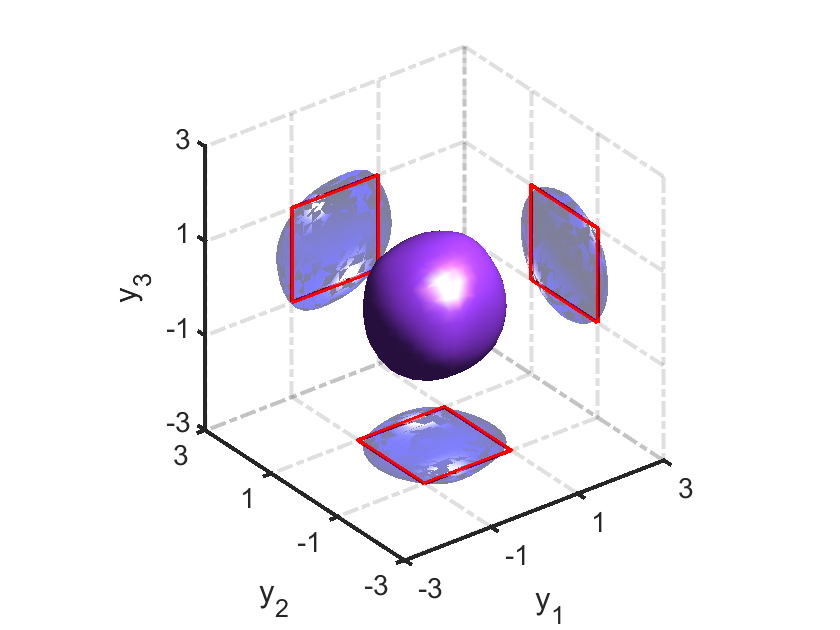}

}
\subfigure[$J=20$]{
\includegraphics[scale=0.16]{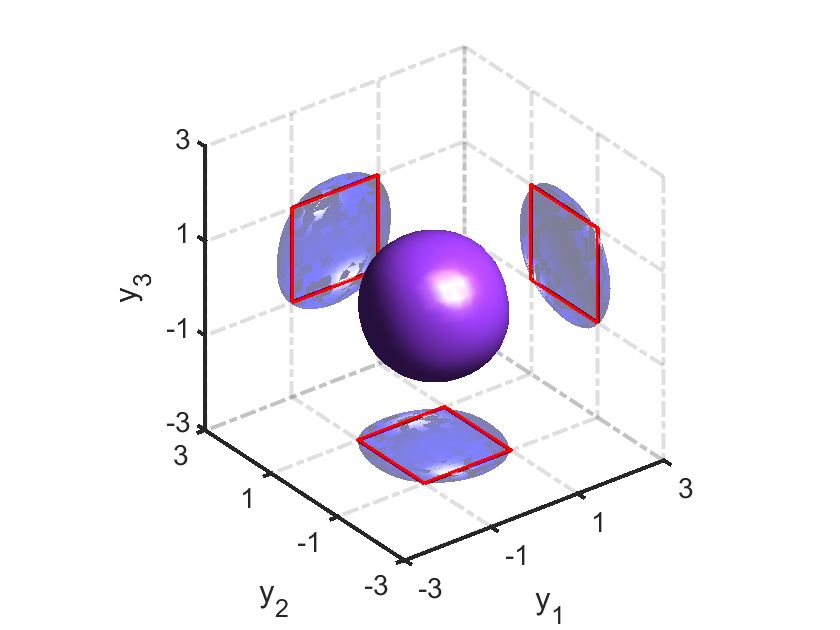}

}
\caption{Reconstructions using far-field data from multiple observation directions for a cube-shaped support centered at the origin.
} \label{fig:3-cube-ma}
\end{figure}

\begin{figure}
\centering
\subfigure[$\hat x=(1,0,0)$]{
\includegraphics[scale=0.2]{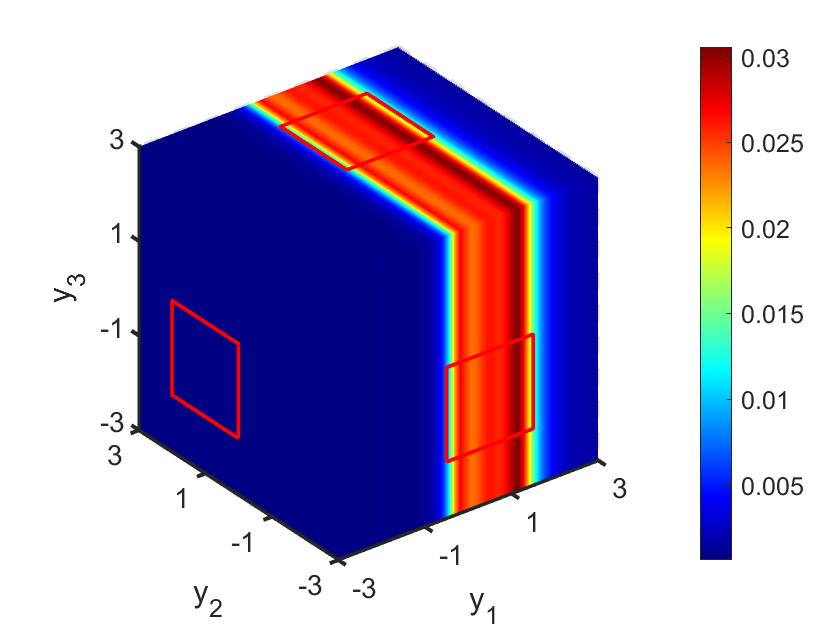}

}
\subfigure[$\hat x =(0,1,0)$]{
\includegraphics[scale=0.2]{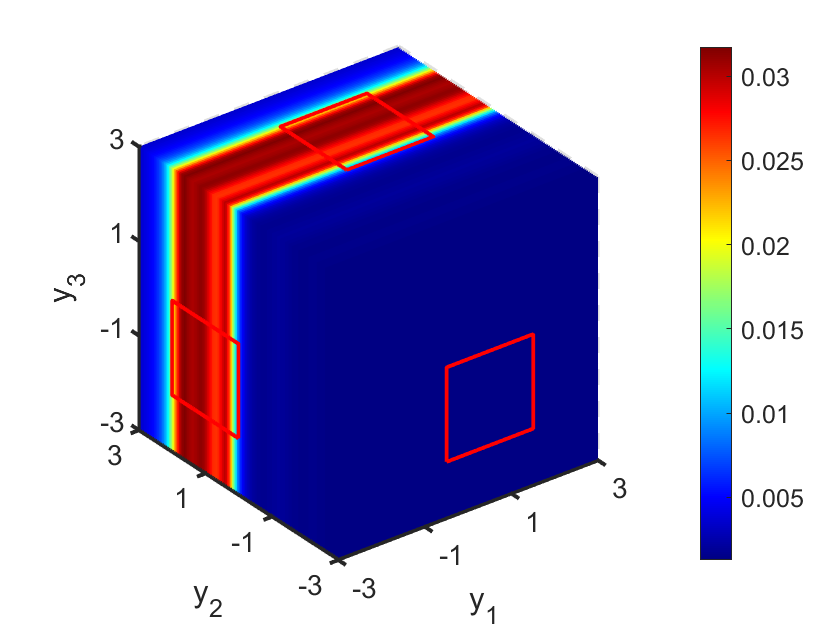}

}
\subfigure[$\hat x =(0,0,1)$]{
\includegraphics[scale=0.2]{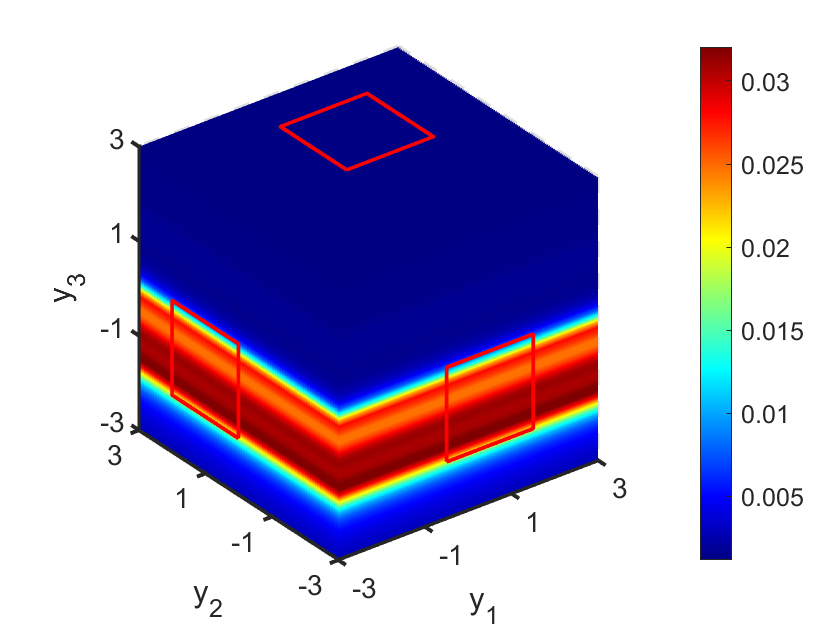}

}
\caption{Reconstructions using far-field data from a single observation direction for a cube-shaped support with an off-origin center.
} \label{fig:3-cube-1b}
\end{figure}

\begin{figure}
\centering
\subfigure[Cube-shaped support]{
\includegraphics[scale=0.16]{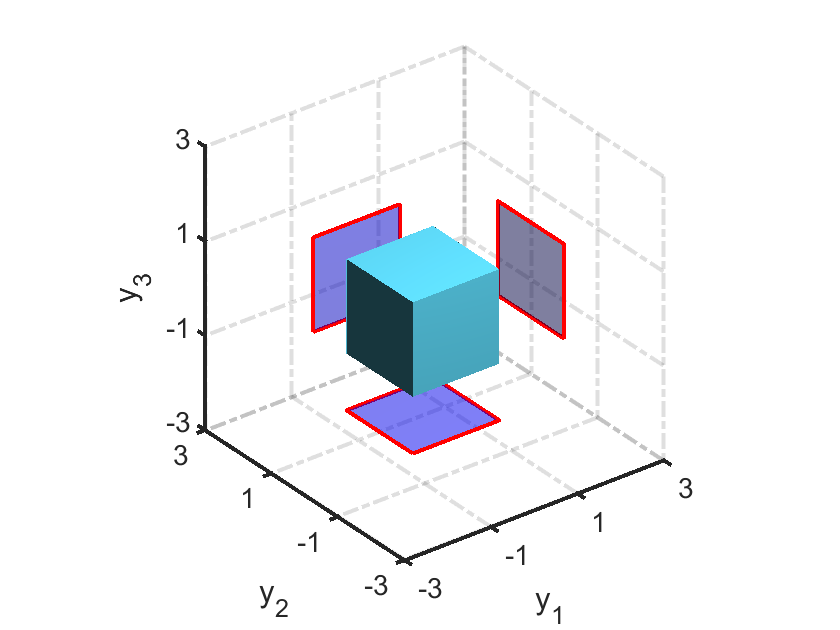}

}
\subfigure[$J=3$]{
\includegraphics[scale=0.16]{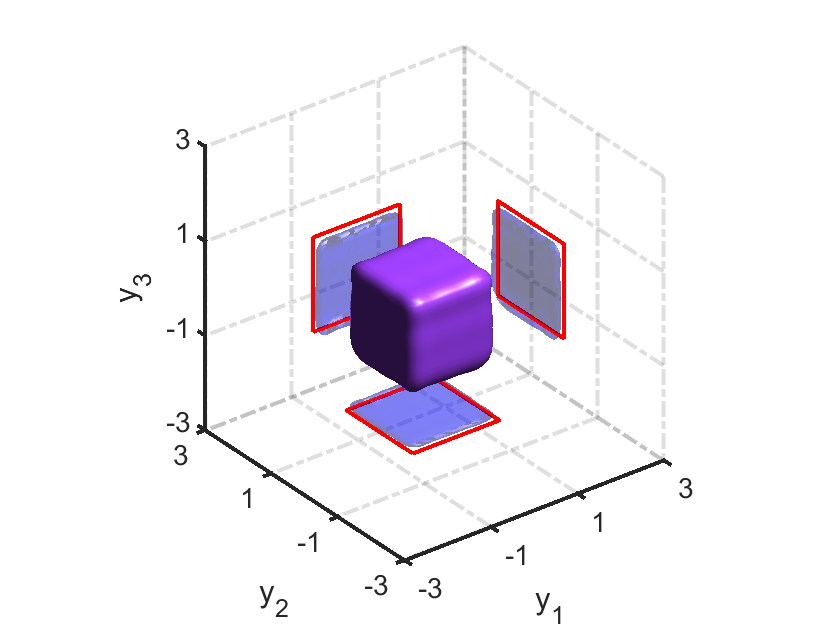}

}
\subfigure[$J=10$]{
\includegraphics[scale=0.16]{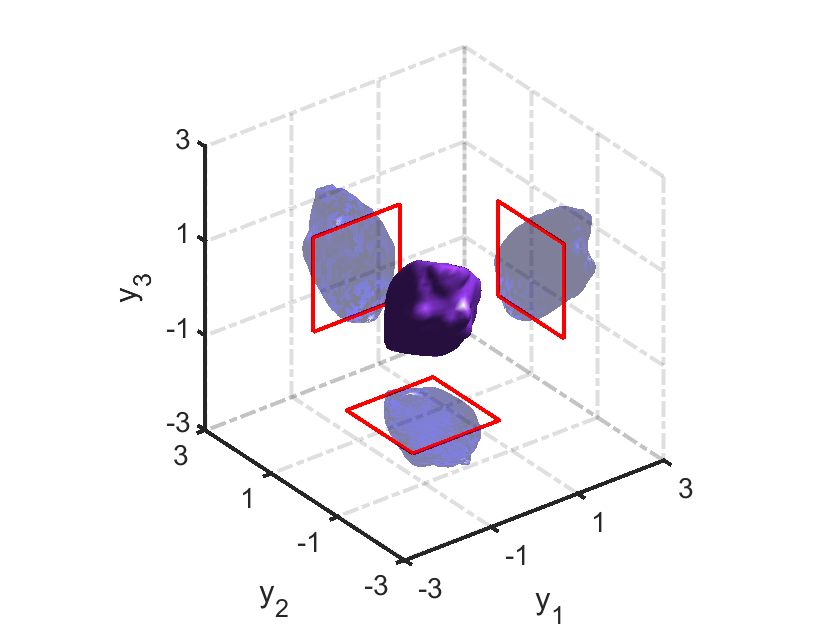}

}
\subfigure[$J=20$]{
\includegraphics[scale=0.16]{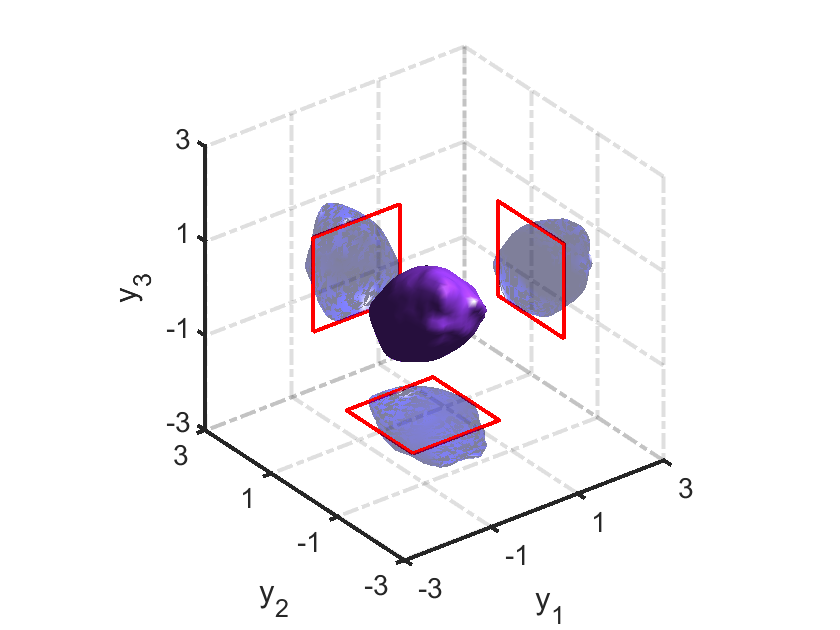}

}


\caption{Reconstructions using far-field data from multiple observation directions for a cube-shaped support with an off-origin center.
} \label{fig:3-cube-mb}
\end{figure}

\begin{figure}
\centering
\subfigure[Ball-shaped support]{
\includegraphics[scale=0.16]{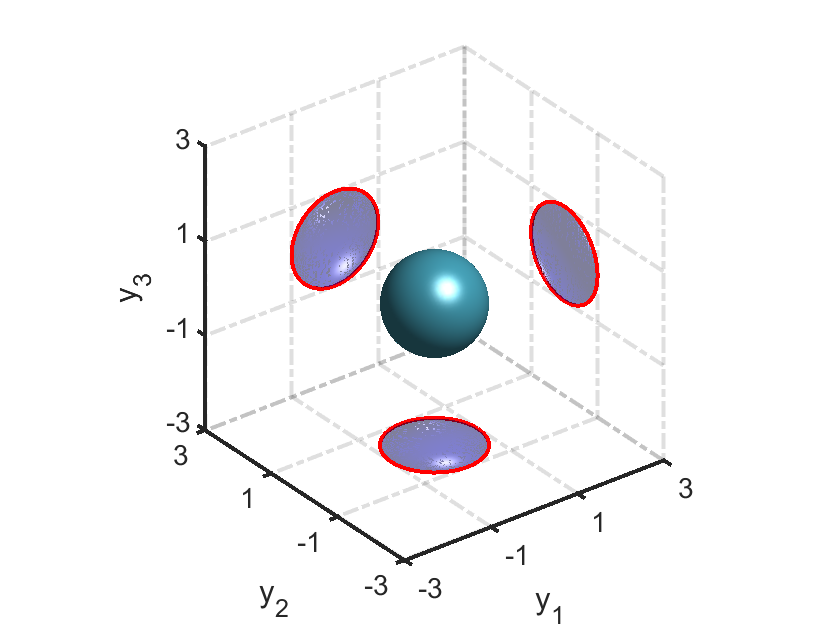}

}
\subfigure[$J=3$]{
\includegraphics[scale=0.16]{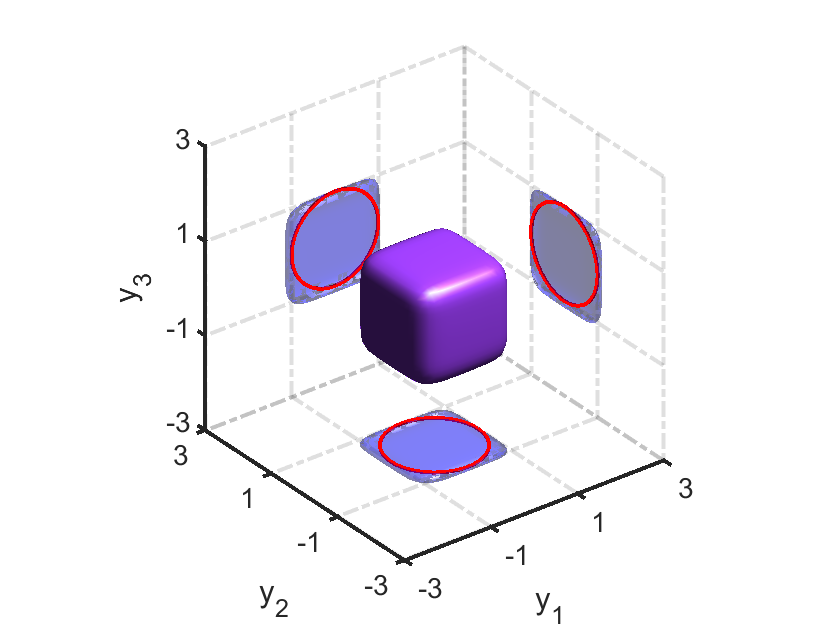}

}
\subfigure[$J=10$]{
\includegraphics[scale=0.16]{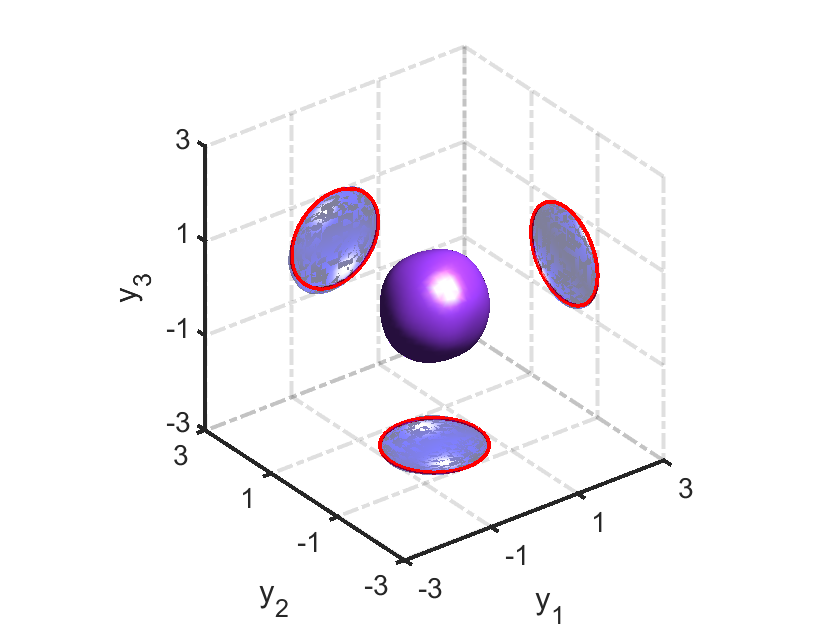}

}
\subfigure[$J=20$]{
\includegraphics[scale=0.16]{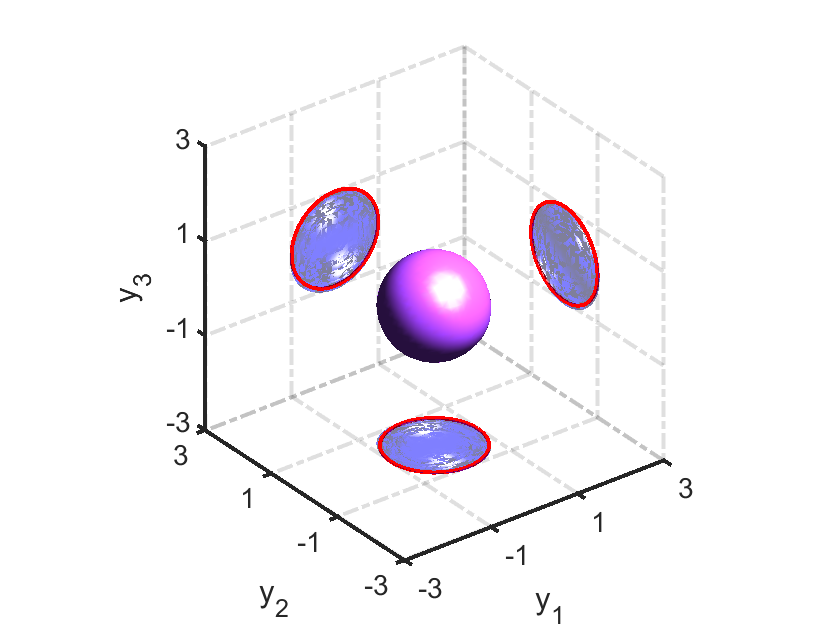}

}
\caption{Reconstructions using far-field data from multiple observation directions for a ball-shaped support centered at the origin.
} \label{fig:3-ball-ma}
\end{figure}

\begin{figure}
\centering
\subfigure[Ball-shaped support]{
\includegraphics[scale=0.16]{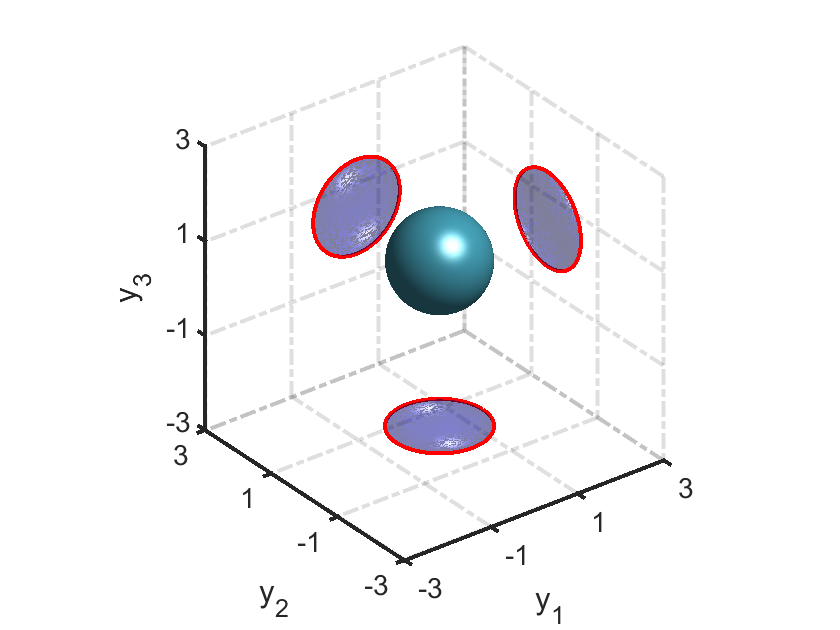}

}
\subfigure[$J=3$]{
\includegraphics[scale=0.16]{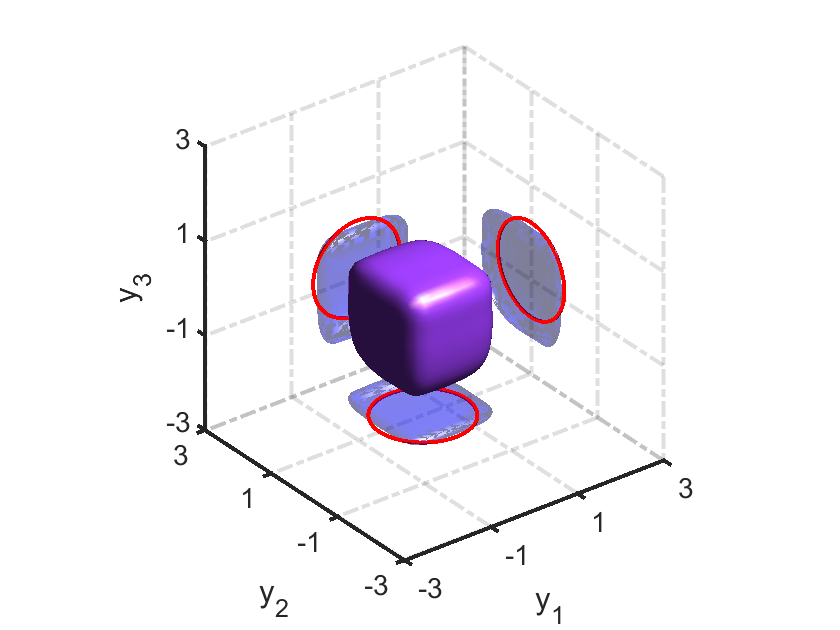}

}
\subfigure[$J=10$]{
\includegraphics[scale=0.16]{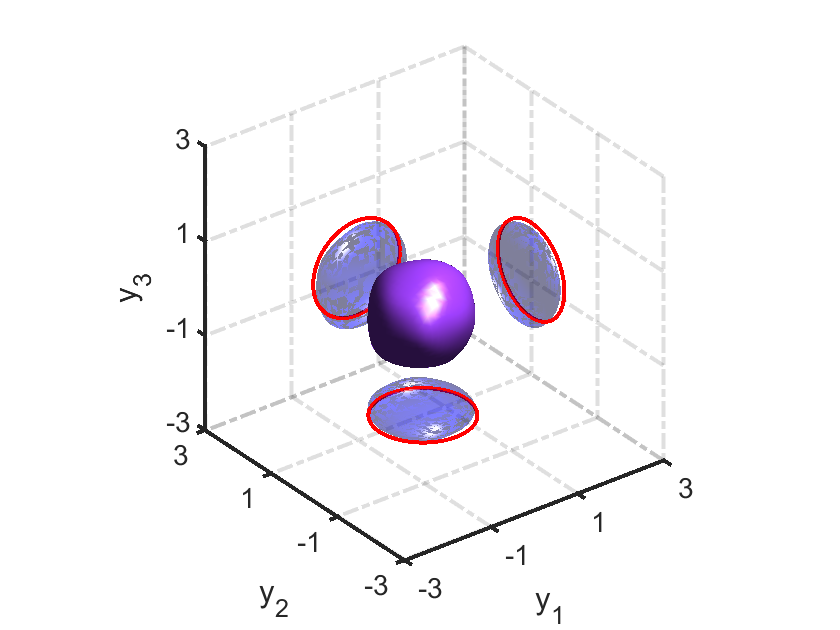}

}
\subfigure[$J=20$]{
\includegraphics[scale=0.16]{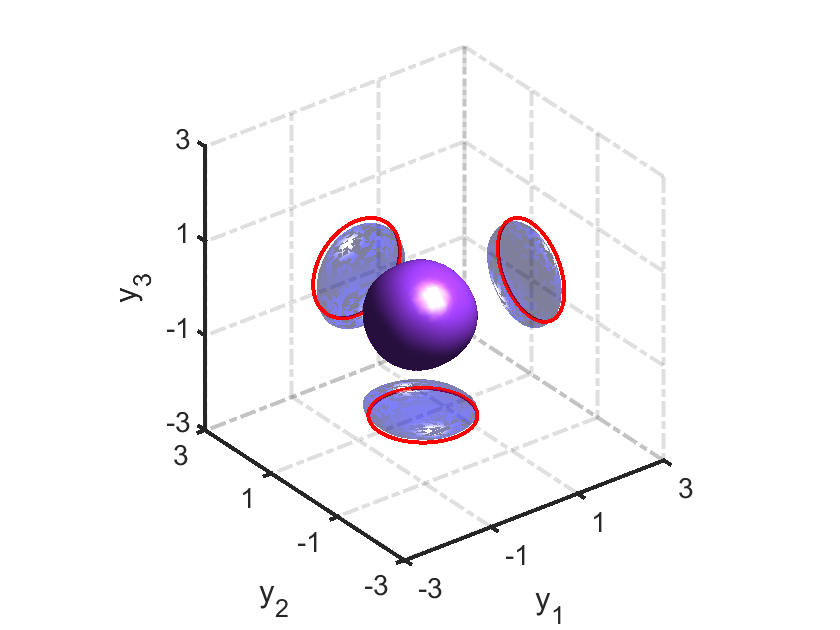}

}
\caption{Reconstructions using far-field data from multiple observation directions for a ball-shaped support with an off-origin center.
} \label{fig:3-ball-mb}
\end{figure}

\begin{figure}
\centering
\subfigure[Ellipsoidal support]{
\includegraphics[scale=0.16]{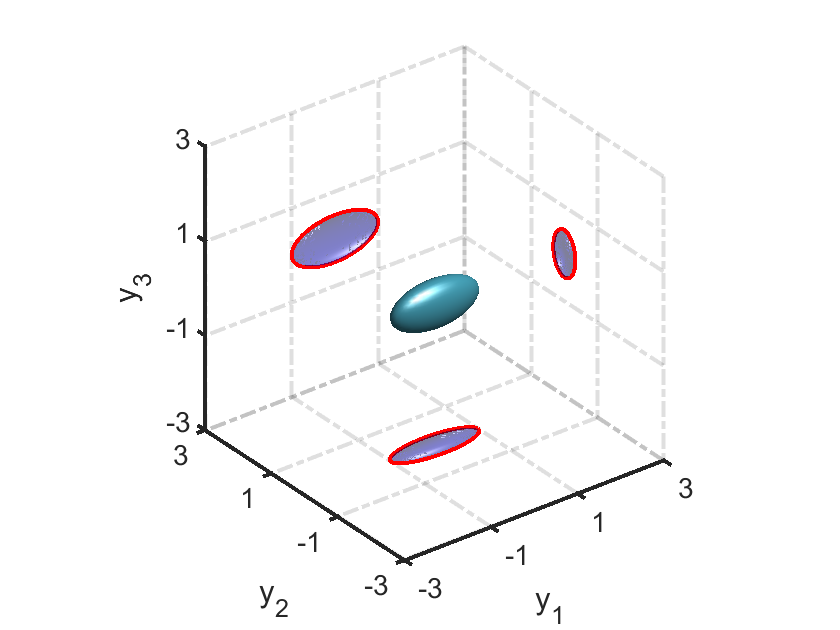}

}
\subfigure[$J=3$]{
\includegraphics[scale=0.16]{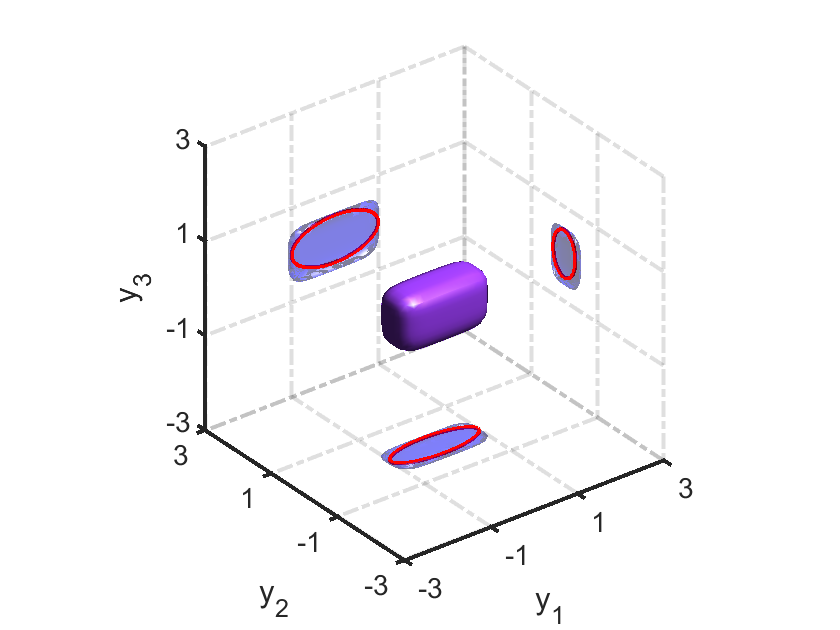}

}
\subfigure[$J=10$]{
\includegraphics[scale=0.16]{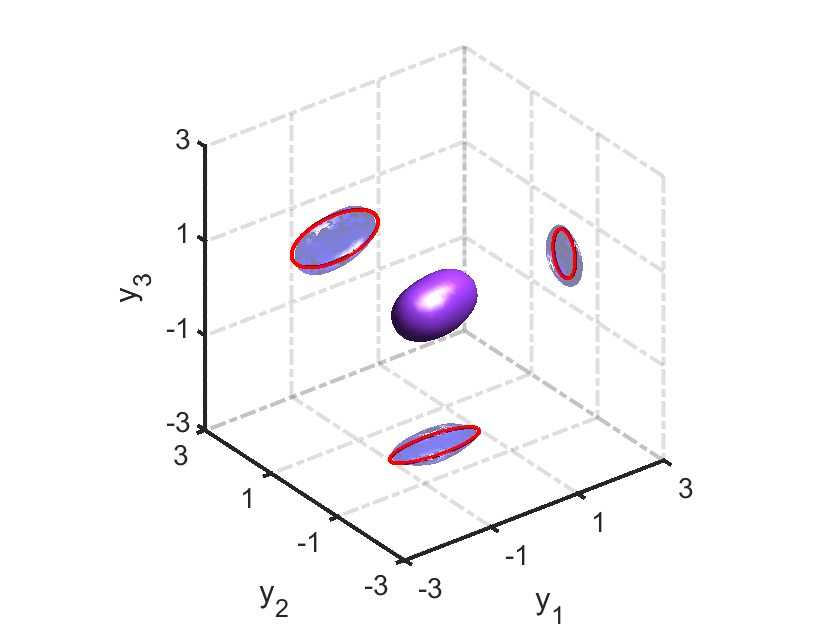}

}
\subfigure[$J=20$]{
\includegraphics[scale=0.16]{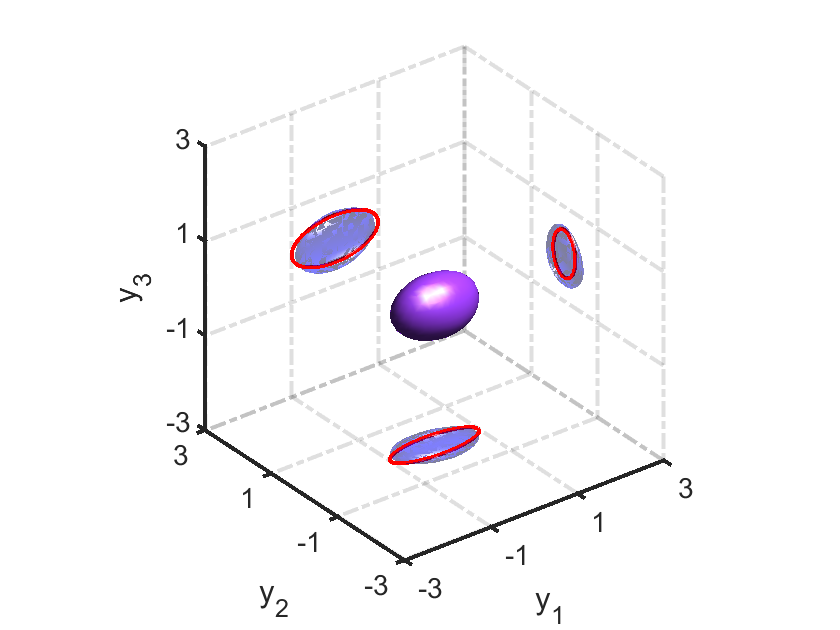}

}
\caption{Reconstructions using far-field data from multiple observation directions for an ellipsoidal support centered at the origin.
} \label{fig:3-ellip-ma}
\end{figure}

\begin{figure}
\centering
\subfigure[Ellipsoidal support]{
\includegraphics[scale=0.16]{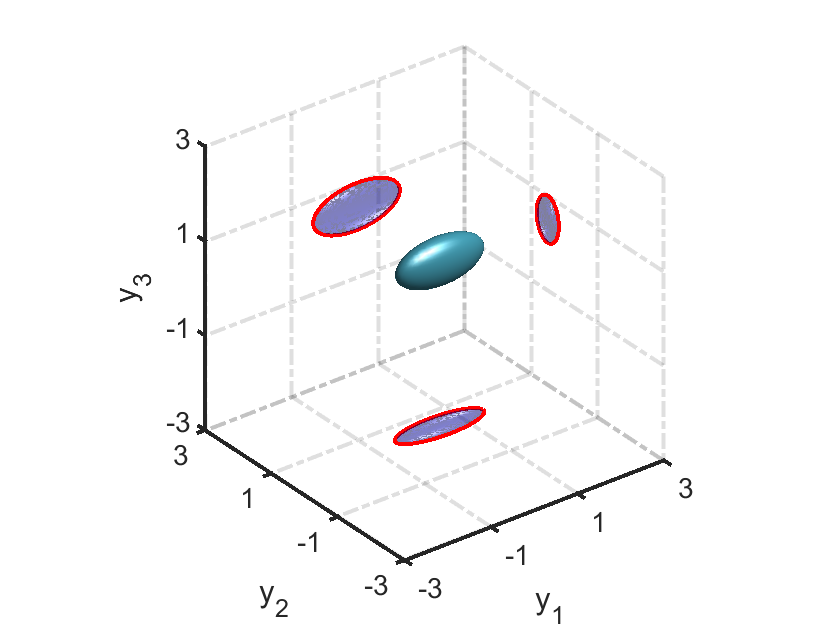}

}
\subfigure[$J=3$]{
\includegraphics[scale=0.16]{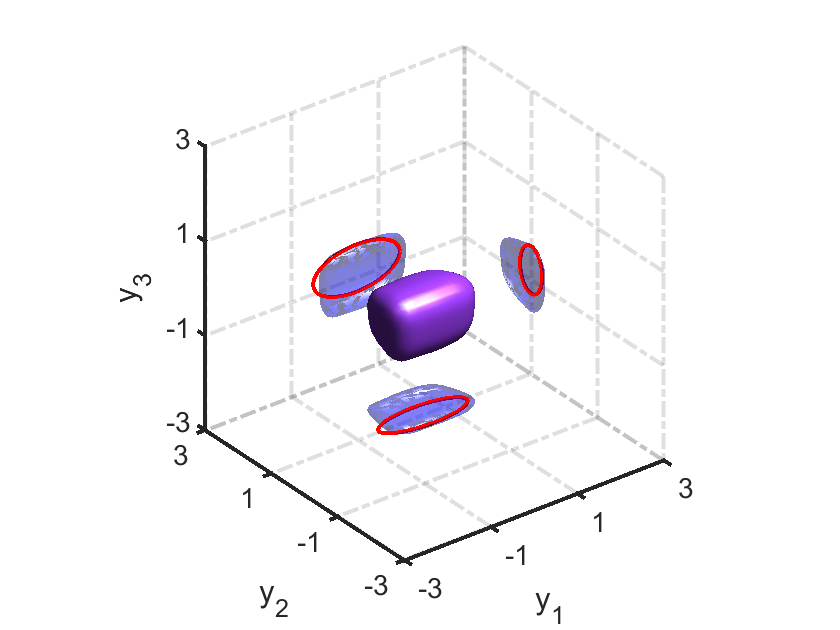}

}
\subfigure[$J=10$]{
\includegraphics[scale=0.16]{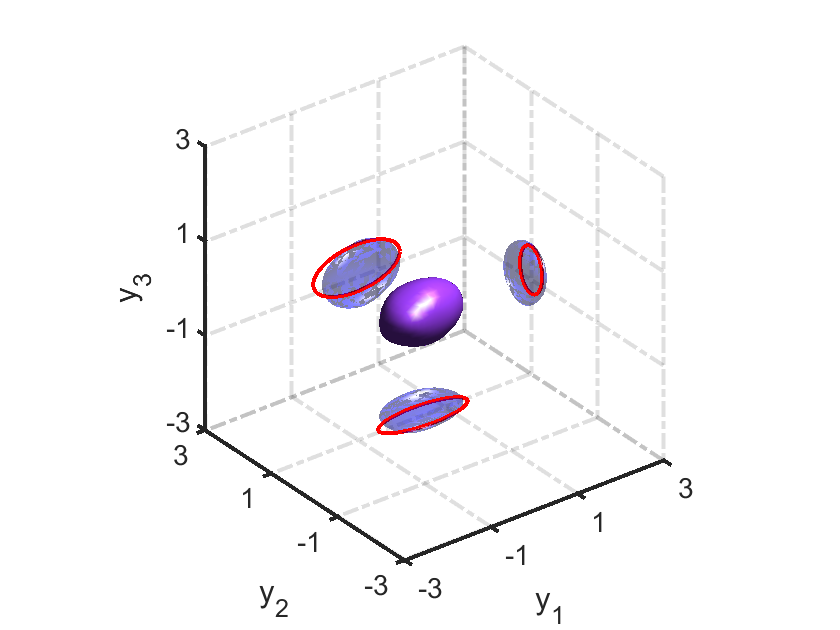}

}
\subfigure[$J=20$]{
\includegraphics[scale=0.16]{Fig_3D/ex3_set1b_m10.png}

}
\caption{Reconstructions using far-field data from multiple observation directions for an ellipsoidal support with an off-origin center.
} \label{fig:3-ellip-mb}
\end{figure}

\section{Conclusion}

In this paper, we consider the inverse problem of determining the uncontrolled noise source for the acoustic wave equation in two and three dimensions. The noise source
is a random process with compact spatial support and will vanish after a finite period of time. Therefore, it is not stationary in time.
By transforming the wave equation into the Helmholtz equation via Fourier transform, we propose an efficient factorization method
to locate the support of the random source based on the multi-frequency far-field correlation data. The proposed method only requires sparse
observation directions and has low computational cost.
Numerical examples in both two and three dimensions are presented to demonstrate the validity and effectiveness of the proposed method. Although this paper concerns the inverse random source scattering problem for the acoustic wave equation, we believe that the proposed framework and methodology can be directly applied to solve inverse random source problems for other wave equations,  such as the electromagnetic and elastic wave equations.  It is a more challenging direction to consider the inverse random medium scattering problem where the medium should be modeled as a random function. We hope to be able to report the progress on these problems in the future.

\appendix
\section{Factorization for the near-field case in three dimensions}
In this section, we will carry the factorization with multi-frequency far-field correlation data to the case with near-field measurements in three dimensions. More precisely,
assuming $D\subset B_R$, the proposed factorization method can be slightly modified to  get an image of the annular region
\[
\tilde K^{(x)}_D:=\{ y\in \mathbb R^3: \inf_{z\in D}|x-z|<|x-y|<\sup_{z\in D}|x-z|\}\subset \mathbb R^3,
\]
for every fixed measurement position $x\in \partial B_R$. Define the near-field operator $\mathcal N^{(x)}: L^2(0,\tau_w)\rightarrow L^2(0,\tau_w)$ by
\begin{eqnarray*}
    &&(\mathcal{N}\phi)(\tau)
    =(\mathcal{N}^{(x,k_0)}\phi)(\tau)   \\
    &=&\int_{0}^{\tau_w} \tilde C(x, \tau_c+\tau-\gamma,k_0)\,\phi(\gamma)\,{\rm d}\gamma, \quad {\tau \in (0,\tau_w)}, \\
    &=& \int_{0}^{\tau_w} \int_{D} \int_{t_1}^{t_2}\int_{t_1}^{t_2} F(t-s)K(y,t)K(y,s)e^{{\rm i}(k_0+\tau_c+\tau-\gamma)t}e^{-{\rm i} k_0 s} \, {\rm d} t {\rm d}s\, \frac{e^{{\rm i}(\tau_c+\tau-\gamma)|x-y|}}{4\pi|x-y|}{\rm d}y\,\phi(\gamma)\,{\rm d}\gamma
\end{eqnarray*}
where $\tilde C(x, \tau, k):= {\mathbb E}[u(x, k+\tau)\overline{u(x, k)}]$. Following the proof of Proposition \ref{Fac-F}, we obtain a factorization of the near-filed operator:
\[
\mathcal N^{(x)}=\tilde {\mathcal L} \tilde {\mathcal T} \tilde {\mathcal L}^*
\]
where
$\mathcal{\tilde L}=\mathcal{\tilde L}_D^{(x,k_0)}:  L^2(D\times[t_1,t_2])\rightarrow  L^2(0,\tau_w)$ is defined by
\begin{eqnarray*}
    (\mathcal{\tilde L}u)(\tau)=\int_{t_1}^{t_2}\int_D e^{-{\rm i} \tau (|x-y|+t)} u(y,t)\,{\rm d}y {\rm d}t,\qquad \tau\in (0, \tau_w),
\end{eqnarray*}
for all $u\in L^2(D\times[t_1,t_2])$, and the middle operator
$\mathcal{\tilde T}: L^2(D\times[t_1,t_2]) \rightarrow L^2(D\times[t_1,t_2])$ is  again a multiplication operator defined by
\begin{eqnarray}
    (\mathcal{\tilde T}\psi)(y,t):=\frac{1}{4\pi|x-y|}\int_{t_1}^{t_2} F(t-s)K(y,t)K(y,s)e^{{\rm i} k_0 (t-s)} e^{{\rm i} \tau_c (t+|x-y|)}\, {\rm d}s\, \psi(y,t).   \nonumber
\end{eqnarray}
Note the adjoint of $\tilde L:L^2(0,\tau_w)\rightarrow L^2(D\times[t_1,t_2])$ is defined by
\begin{eqnarray*}
    (\mathcal{\tilde L^*}\phi)(y,t)=\int_{t_1}^{t_2}\int_D e^{{\rm i} \gamma (|x-y|+t)} \phi (\gamma)\,{\rm d}\gamma,\qquad \phi\in L^2(0, \tau_w ).
\end{eqnarray*}
Choose the test function:
\[
\tilde \phi_{y,\epsilon}^{(x)}(\tau)=\frac{1}{(t_2-t_1)|B_\epsilon(y)|} \int _{t_1}^{t_2} \int _{B_\epsilon(y)} \frac{e^{\rm i\tau(|x-z|+t)}}{4\pi|x-z|} \,{\rm d}z{\rm d}t\,\quad \tau \in (0, \tau_w).
\]
and introduce the indicator function
\[
\tilde {\mathcal I}^{(x)}(y):=\left[\sum_{n=1}^{n=\infty}\frac{|\langle \tilde \phi^{(x)}_{y,\epsilon}, \tilde \psi^{(x)}_n\rangle|^2_{L^2(0,\tau_w)}}{|\tilde \lambda_n^{(x)}|}\right]^{-1}, \quad y \in \mathbb R^3,
\]
where  $(\tilde \lambda_n^{(x)}, \tilde \psi^{(x)}_n)$ is an eigensystem of the near-field operator $\mathcal N^{(x)}_\#$. As a counterpart to $\mathcal I$, we have the following result  for the near-field case.
\begin{corollary}
(i) $\mathrm{Range}[ (\mathcal {N}^{(x)})_\#^{1/2}]=\rm {Range}( \mathcal{\tilde L}^{(x)})$ for all $x\in \partial B_R$.\\
(ii)If $y\in\overline{\tilde K^{(x)}_D}$, the test function $\tilde \phi _{y,\epsilon}^{(x)}\in \mathrm{Range}(\tilde {\mathcal L}^{(x)})$. If $y\notin\overline{\tilde K^{(x)}_D}$, the test function $\tilde \phi _{y,\epsilon}^{(x)}\notin \mathrm{Range}(\tilde {\mathcal L}^{(x)})$.Moreover, it holds that $\tilde {\mathcal I} ^{(x)}(y)>0 $ if $y\in \overline{\tilde K^{(x)}_D}$ and $\tilde {\mathcal I} ^{(x)}(y)=0 $ if $y\notin \overline{\tilde K^{(x)}_D}$.
\end{corollary}

\end{document}